\documentclass[11pt]{article}

\usepackage[a4paper,left=2.5cm,right=2.5cm,top=2cm,bottom=2.5cm]{geometry}
\usepackage[utf8]{inputenc}
\usepackage[T1]{fontenc}
\usepackage[english]{babel}
\usepackage{amsthm}
\usepackage{amsmath}
\usepackage{amssymb}
\usepackage{mathrsfs}
\usepackage{setspace}
\usepackage{float}
\usepackage{multirow}
\usepackage{stmaryrd}
\usepackage{color}
\usepackage{wrapfig}
\usepackage{pifont}
\usepackage{array}
\usepackage[colorlinks=true,linkcolor=ocre,citecolor=ocre]{hyperref}
\usepackage{dsfont}
\usepackage{upgreek}
\usepackage{nicefrac}
\usepackage{tikz}{}
\usepackage{gensymb}
\usepackage{FiraSans}

\usetikzlibrary{arrows}
\usetikzlibrary{shapes}
\usetikzlibrary{positioning}
\usetikzlibrary{decorations.pathmorphing}

\usepackage{graphicx}
\usepackage{caption}
\renewenvironment{proof}{\noindent{\sffamily{\textbf{Proof :}}}}{\begin{flushright}$\square$\end{flushright}}

\newcommand{\IE}{\mathbb{E}}

\newcommand{\IR}{\mathbb{R}}

\newcommand{\IT}{\mathbb{T}}

\newcommand{\drm}{\mathrm d}

\newcommand{\CT}{\mathcal K}

\newcommand{\CR}{\mathcal T}

\newcommand{\CC}{\mathcal C}

\newcommand\N{\mathbb{N}}
\newcommand\T{\mathbb{T}}
\newcommand\Z{\mathbb{Z}}

\newcommand\R{\mathbb{R}}
\newcommand\C{\mathbb{C}}

\newcommand\E{\mathbb{E}}

\newcommand{\dd}{\mathrm{d}}

\newcommand\A{\mathfrak{A}}

\newcommand\eps{\varepsilon}
\renewcommand{\Re}{\operatorname{Re}}
\renewcommand{\Im}{\operatorname{Im}}

\newcommand{\Norm}[2]{\|#1\|\left.\vphantom{T_{j_0}^0}\!\!\right._{#2}}

\newcommand{\widebar}[1]{\mkern 1.5mu\overline{\mkern-1.5mu#1\mkern-1.5mu}\mkern 1.5mu}

\definecolor{toon}{RGB}{71,145,67} 

\definecolor{ocre}{RGB}{64,123,121}
\definecolor{S}{rgb}{0.0,0.5,0.0}

\newcounter{item}
\numberwithin{item}{section}

\newtheorem{theorem}[item]{\sffamily Theorem}
\newtheorem{definition}[item]{\sffamily Definition}
\newtheorem{proposition}[item]{\sffamily Proposition}
\newtheorem{lemma}[item]{\sffamily Lemma}
\newtheorem{corollary}[item]{\sffamily Corollary}

\newtheorem*{theorem*}{\sffamily Theorem}
\newtheorem*{definition*}{\sffamily Definition}
\newtheorem*{proposition*}{\sffamily Proposition}
\newtheorem*{lemma*}{\sffamily Lemma}
\newtheorem*{corollary*}{\sffamily Corollary}

\usepackage[explicit]{titlesec}
\titleformat{\section}{\centering\Large\bfseries}{\thesection \ --}{0.7em}{\Large\bfseries #1}
\titleformat{\subsection}{\centering\large\bfseries}{\thesubsection \ --}{0.4em}{\large\bfseries #1}
\titleformat{\subsubsection}{\centering\bfseries}{\thesubsubsection \ --}{0.4em}{\bfseries #1}

\let\emph\relax
\DeclareTextFontCommand{\emph}{\bfseries\em}

\title{\bfseries Scattering, random phase and wave turbulence}
\author{Erwan FAOU and Antoine MOUZARD}
\date{}

\begin{document}

\maketitle
\abstract{We start from the remark that in wave turbulence theory, exemplified by the cubic two-dimensional Schr\"odinger equation (NLS) on the real plane, the regularity of the resonant manifold is linked with dispersive properties of the equation and thus with scattering phenomena. In contrast with classical analysis starting with a dynamics on a large periodic box, we propose to study NLS set on the real plane using the dispersive effects, by considering the  time evolution operator in various time scales for deterministic and random initial data. By considering periodic functions embedded in the whole space by gaussian truncation, this allows explicit calculations and we identify two different regimes where the operators converges towards the kinetic operator but with different form of convergence.}
\vspace{0.5cm}

\section{Introduction and main results}

The theory of wave turbulence aims at describing the nonlinear interaction of waves outside thermal equilibrium, and also as the statistical behavior of a system of random nonlinear waves. Following Boltzmann's kinetic theory of gases, a wave kinetic theory was developped during the last century. It was first studied by Peierls for the description of anharmonic crystals in the 1930's and then by Hasselman for water waves equations and by Zakharov for out-of-equilibrium turbulent systems in the 1960's. As with Boltzmann's kinetic theory, the idea is to describe weak interactions of a large number of waves and the main mathematical contributions concern the case of weakly nonlinear equations with small nonlinearity. We refer to Zakharov, L'vov and Falkovich \cite{ZLF92} and Nazarenko \cite{Naz1} for a complete description of the theory of kinetic wave turbulence. 

\medskip

The kinetic theory of wave turbulence is not yet understood within a complete rigorous mathematical picture however spectacular recent progresses have been made recenlty, see \cite{DH1,DH2,CoGe20,BGHS,LuSpohn} and also \cite{dymovkuksin1,dymovkuksin2,faou1} for problems with random forcing. An important example is given by the cubic nonlinear Schr\"odinger equation
\begin{equation}\tag{NLS}
i\partial_tu=-\Delta u+\varepsilon|u|^2u
\end{equation}
on $\T_L^d$ the periodic box of size $L\gg1$ with nonlinearity strength $|\varepsilon|\ll1$. In Fourier variables, this rewrites as the system of coupled equations
\begin{equation*}
i\partial_tu_K=\omega_K u_K+\varepsilon\sum_{K=K_1-K_2+K_3}u_{K_1}\overline{u_{K_2}}u_{K_3}
\end{equation*}
for $K\in\Z_L^d$ the lattice of mesh $L^{-1}$ and $\omega_K=|K|^2$ the dispersion relation. The study of nonlinear dispersive equations on compact domains is very rich and a large number of phenomena can occur, both for deterministic or random initial data. The main difference in comparison with the equation on the full space is that the dispersion of the solution does not imply decay in space, both for linear and nonlinear equations. In particular, all asymptotic stability results around equilibrium are not valid anymore. For example, Faou, Germain and Hani \cite{fgh} found coherent dynamics in this large volume and weak nonlinearity regime for NLS in two dimensions, described by the Continuous Resonant (CR) equation
\begin{equation}\tag{CR}
i\partial_tg=\CR(g)
\end{equation}
with the operator
\begin{equation}
\label{CR}
\CR_k(g)=\int_{-1}^1\int_{\R^2}g(k+\lambda z)\widebar g(k+\lambda z+z^\bot)g(k+z^\bot)\dd\lambda\dd z
\end{equation}
for $k\in\R^2$ and $z^\bot$ the rotation of $z$ by the angle $\frac{\pi}{2}$. This equation appears as the continuous limit of the system of equations given by \eqref{NLS} in the Fourier variables and the resonant system appears as its first Birkhoff normal form approximation. 

\medskip

In kinetic theory of wave turbulence, the idea is to start with a random state outside the thermal equilibrium and to describe the evolution of the covariance of the Fourier coefficients $(u_K)_{K\in\Z_L^d}$. Two assumptions are usually made for the randomness in the initial data, that is Random Phase (RP) or Random Phase and Amplitudes (RPA). The first assumption amounts to having the angles of the Fourier coefficients to be independent and identically distributed uniform random variables on the unit circle. For the second assumption, the amplitudes of the Fourier coefficients are also supposed to be independent and identically distributed random variables. In this case, the wave kinetic theory predicts that the variance of the Fourier coefficients are well-approximated in the limit by the Wave Kinetic (WK) equation. For NLS, this equation is
\begin{equation}\tag{WK}
\partial_tn=\CT(n)
\end{equation}
with the kinetic operator
\begin{equation*}
\CT_k(n)=\int_{\substack{k=k_1-k_2+k_3\\\Delta\omega_{kk_1k_2k_3}=0}}n(k)n(k_1)n(k_2)n(k_3)\left(\frac{1}{n(k)}-\frac{1}{n(k_1)}+\frac{1}{n(k_2)}-\frac{1}{n(k_3)}\right)\drm k_1\drm k_2\drm k_3, 
\end{equation*}
where $\Delta\omega_{kk_1k_2k_3}=\omega_k-\omega_{k_1}+\omega_{k_2}-\omega_{k_3}$ is the relation of resonance. In particular, the first term corresponds to the operator $\CR$ from the CR equation. One is also interested in the propagation of chaos, that is to understand if the independence of the Fourier coefficients for the initial data is conserved. A rigorous comprehension of this phenomena is a hard question and very recent progress were made by Deng and Hani, see \cite{DH1,DH2,DH3,DH4} and references therein. They consider the expansion with respect to the nonlinearity up to arbitrary order where the coefficients associated to $\lambda^n$ are given by $(2n+1)$-linear functionnals of the initial data with a tree-like structure. Taking the variance yields an expansion where the coefficient are given by Feynman diagrams using the RP or RPA assumption and they are able to identify the important terms in the asymptotic behavior. It is given by
\begin{equation*}
\sum_{K=K_1-K_2+K_3}n_Kn_{K_1}n_{K_2}n_{K_3}\left(\frac{1}{n_K}-\frac{1}{n_{K_1}}+\frac{1}{n_{K_2}}-\frac{1}{n_{K_3}}\right)\left|\frac{\sin(\pi t\Delta\omega_{KK_1K_2K_3})}{\pi t\Delta\omega_{KK_1K_2K_3}}\right|^2
\end{equation*}
which is a convergent Riemann sum for $t\ll L^2$ in dimension $d\ge3$ using number theoretic results. Then this localizes on the resonant manifold because of the time-dependent term for $t\gg1$ giving the wave kinetic operator $\CT$. While this is the effect of quasi-resonances as the main part comes from $(K_1,K_2,K_3)\in\Z_L^2$ such that $|\Delta\omega_{KK_1K_2K_3K}|$ is small, the result from Faou, Germain and Hani for deterministic initial data is a consequence of analysis of the exact resonances, that is $\Delta\omega_{KK_1K_2K_3}=0$ on a discrete lattice in the continuous limit. Finally, these results are not global in time and occur in a timeframe depending on the parameters.

\medskip


The theory of wave turbulence can be used to describe different models such as waves in the ocean and many experiments are made in order to observe similar behaviours. In particular, the system is always in a finite box which has to be large depending on the usual scale of the nonlinear interaction under observation. This explains why the limit kinetic models are posed on functions depending on a continuous set of frequencies. Motivated by this remark, we consider in this work the spatial localization of periodic functions with large period as initial data, that we embed into an equation with continuous spectrum. We consider the two-dimensional cubic Schr\"odinger equation 
\begin{equation}\label{NLS}\tag{NLS}
i\partial_tu=- \Delta u+|u|^2u
\end{equation}
on the full space $\R^2$. We then propose a new family of initial data and observe a similar behavior compared to the previous models set on large tori, coming from the quasi-resonances for time $t\ll L^2$ and exact resonances for $t\gg L^2$. For deterministic data, the kinetic operator is $\CR$ from the CR equation while for initial data with random phase, the variance is described by $\CT$ from the WK equation. Hence  these operators appear naturally in the time expansion of the scattering evolution operator.

\medskip 

With $f_k=\widehat{f}(k)$ the Fourier transform of $f$ evaluated at $k\in\IR^2$, the equation \eqref{NLS} is written
\begin{equation*}
i \partial_t u_k = \omega_k u_k + \int_{k = \ell - m + j} u_\ell \widebar u_m u_j, 
\end{equation*}
where $\omega_k := k_1^2 + k_2^2$ for $k = (k_1,k_2) \in \R^2$. One can consider $v(t)=e^{-it\Delta}u(t)$ which satisfies the equation
\begin{equation*}
i\partial_tv_k=\int_{k=\ell-m+j}e^{-it\Delta\omega_{k\ell mj}}v_\ell \bar v_m v_j
\end{equation*}
with
\begin{equation*}
\Delta\omega_{k\ell mj}=\omega_{k}+\omega_{m}-\omega_\ell-\omega_{j}=|k|^2+|m|^2-|\ell|^2-|j|^2.
\end{equation*}
The solution $u$ is recovered from $v$ with $u_k(t)=e^{-it\omega_k}v_k(t)$ and this motivates the study of the trilinear operator 
\begin{equation}
\label{trilin}
R_k(t,u,v,w):=\int_{k=\ell-m+j}e^{it \Delta\omega_{k\ell mj}}u_\ell \bar v_m w_j.
\end{equation}
The co-area formula states that 
\begin{equation*}
R_k(t,u,v,w)=\int_{\R}e^{it \xi}\left(\int_{S_{k}(\xi)}u_\ell\widebar v_m w_j\,\dd S_{k}(\xi)\right)\dd\xi
\end{equation*}
where for $\xi\in\IR$, the set $S_k(\xi)$ is given by
\begin{equation*}
S_{k}(\xi) = \big\{(j,\ell,m)\in(\R^2)^3\ ;\ k+m-\ell-j=0\quad\mbox{and}\quad\Delta\omega_{k\ell mj}=\xi\big\}
\end{equation*}
and $\dd S_k(\xi)$ is the associated {\em microcanonical} measure. For regular functions $u,v$ and $w$, we have 
\begin{equation*}
\widehat R_k (\xi,u,v,w) = 2\pi \int_{S_{k}(\xi)} u_\ell\widebar v_m w_j \, \dd S_{k}(\xi)
\end{equation*}
which has regularity in $\xi$ related to the decreasing properties of $t\mapsto R_k(t,u,v,w)$. As proved in proposition \ref{LemCR}, the resonant case $\xi=0$ corresponds to the trilinear CR operator introduced by Faou, Germain and Hani in \cite{fgh}, that is
\begin{equation*}
\int_{S_{k}(0)} 
u_\ell \widebar v_m w_j \, \dd S_{k}(0)  = \frac12 \int_{\R} \int_{\R^2}
u_{k + a}\widebar v_{k + \lambda a^\perp} w_{k + a + \lambda a^\perp} \dd \lambda \dd a
= \mathcal{T}_k(u,v,w).
\end{equation*}
This is also the same microcanonical measure as introduced by Dymov and Kuksin in \cite{dymovkuksin1} up to a multiplicative factor. With this approach, the regularity of the resonant manifold appears to be related to the scattering properties of the equation.

\medskip

In the case of \eqref{NLS}, this is well-known and the solution scatters at $t = \pm \infty$, that is the solution $v(t) = e^{- i t \Delta} u(t)$ has  limits $v_{\pm \infty}$ when $t \to \pm \infty$, for initial data $u(0)=v(0)$ in the space 
\begin{equation*}
\Sigma=\big\{\varphi(x)\in H^1(\R^2)\ ;\ |x|\varphi(x)\in L^2(\R^2)\big\}.
\end{equation*}
This holds for any $\varphi \in \Sigma$ while $\Norm{\varphi}{\Sigma}$ has to be small enough in the focusing case, that is with the other sign in front of the non linearity. The main argument for \eqref{NLS} lies on a pseudoconformal conservation law argument, see  Ginibre and Velo \cite{GV1}, Tsutsumi \cite{tsutsumi} or Cazenave \cite[Theorem 7.2.1]{cazenave}. As a consequence, we can study the application 
\begin{equation*}
u(0)  = \varphi\mapsto U(t,\varphi) = v(t)
\end{equation*}
which is well defined for $t \in \R$ and $\varphi \in \Sigma$. In the following, we consider the case where $\varphi$ is small in $\Sigma$, which is equivalent to having a small nonlinearity. In particular, the sign in front of the cubic term is not important here. Using the result of Carles and Gallagher \cite{CaGa09}, the application $U(t,\varphi)$ is analytic with respect to $\varphi\in\Sigma$ hence we can consider the expansion
\begin{equation*}
v_k(t)=\varphi_k+\sum_{n \geq 1} (-i)^nV_k^n(t)
\end{equation*}
where $\Norm{V^n(t)}{\Sigma} \leq C  \Norm{\varphi}{\Sigma}^{2n + 1}$ for all $n\in\N$ and $t\in\R$. The second order expansion is given by
\begin{equation*}
V_k^1(t) = \int_0^t R_k(s, \varphi,\varphi,\varphi) \dd s
\end{equation*}
and
\begin{equation*}
V_k^2(t) = 2\int_0^t \int_{0}^{s} R_k(s, \varphi,\varphi, R(s',\varphi,\varphi,\varphi)) \dd s' \dd s+\int_0^t \int_{0}^{s} R_k(s, \varphi, R(s',\varphi,\varphi,\varphi),\varphi) \dd s' \dd s.
\end{equation*}
The proof of this result is based on Strichartz estimates, and conservation laws of the Schrödinger equation. The link between the regularity of the resonant manifolds and the scattering effect is essentially expressed by dispersive estimates, as we illustrate in Section \ref{SectionScatt}.

\medskip

We propose a new family of initial data for a better understanding of wave turbulence. Given two parameters $h,L>0$, we consider
\begin{equation*}
\varphi(x)=\frac{1}{(2\pi)^2}e^{-\frac{1}{2}h^2|x|^2}\sum_{K\in\Z_L^2}\eta_Ke^{iK\cdot x}=e^{-\frac{1}{2}h^2|x|^2}F_L(x)
\end{equation*}
for $x\in\R^2$, that is essentially a $(2\pi L)$-periodic function $F_L$ embedded in $\Sigma$ by Gaussian truncation, assuming enough decay for $(\eta_K)_{K\in\Z_L^2}$. We are interested in the observation of a large number of large period hence in the limit $L\gg1$ and $hL\ll1$. Assuming that $\eta_K=\eta(K)$ with $\eta:\R^2\to\C$ a smooth decreasing function, our goal is to describe the asymptotic of the solution to \eqref{NLS} on a timescale depending on the parameters in the limit $L\gg1$ and $h\ll1$. Up to a multiplication by a small factor depending on $h$ and $L$, we can use the previous expansion with bounds uniform in time and describe $V_k^1(t)$ and $V_k^2(t)$ for such initial data. In frequency, we have
\begin{equation*}
\widehat{\varphi}(k)=\frac{1}{2\pi h^2}\sum_{K\in\Z_L^2}\eta_Ke^{-\frac{|k-K|^2}{2h^2}}
\end{equation*}
which converges ah $h$ goes to $0$ to the sum of Dirac
\begin{equation*}
\sum_{K\in\Z_L^2}\eta_K\delta_0(k-K).
\end{equation*}
To deal with almost Dirac functions from the limit $h\ll1$ and with the continuous limit $L\gg1$, we introduce another scale of observation $\sigma>0$ and for a function $v$ which is expected to be close to a $2\pi L$-periodic function, we define the {\em coarse grained} quantity in frequency
\begin{equation*}
\langle v\rangle_{K,\sigma} =   \int_{\R^2}e^{- \frac{1}{2 \sigma^2} |k- K|^2} \widehat v(k) \dd k
\end{equation*}
for $K\in\Z_L^2$. This scale of observation will in particular be taken such that $h\ll\sigma\ll\frac{1}{L}$, see \eqref{regime1} below for the precise scaling assumption on $(h,L,\sigma)$. Recall that
\begin{equation*}
\widehat R_k(\xi)=\int_{\substack{k=k_1-k_2+k_3\\\Delta\omega_{kk_1k_2k_3}=\xi}}\eta(k_1)\overline{\eta(k_2)}\eta(k_3)\dd k_1\dd k_2\dd k_3
\end{equation*}
for $k\in\IR^2$ and $\widehat R_k(0)=\CR_k(\eta)$.


\begin{theorem}\label{TheoremDeter}
Let $(h,L,\sigma)$ be in the asymptotic scaling \eqref{regime1}, $K\in\Z_L^2$, $u$ the solution to \eqref{NLS} with initial data $u(0)=\varepsilon\varphi$, $\delta>0$ and assume $\varepsilon\ll\frac{h^2}{L}$. Then $v(t)=e^{-it\Delta}u(t)$ satisfies
\begin{equation*}
\langle v(t)\rangle_{K,\sigma}=\frac{\sigma^2\varepsilon}{\sigma^2+h^2}\eta(K)-i\pi\frac{\varepsilon^3L^4}{(2\pi)^4}\CR_K(\eta)+\frac{\varepsilon^3L^4}{(2\pi)^4}\int_\R\frac{\widehat R_K(\xi)-\CR_K(\eta)}{\xi}\dd\xi+o(\varepsilon^3L^4)
\end{equation*}
for $L^\delta\le t\le L^{1-\delta}$. For $L^{2+\delta}\le t\le\frac{1}{hL^\delta}$, we have 
\begin{equation*}
\langle v(t)\rangle_{K,\sigma}=\frac{\sigma^2\varepsilon}{\sigma^2+h^2}\eta(K)-i\pi\frac{2t\varepsilon^3L^2\log(L)}{\zeta(2)(2\pi)^4}\CR_K(\eta)+o\big(t\varepsilon^3L^2\log(L)\big).
\end{equation*} 
\end{theorem}

\medskip

Using the previous expansion for times $t\ll\frac{1}{h}$, we prove that the solution $v$ is described by the discrete sum
\begin{equation*}
\langle v(t)\rangle_{K,\sigma}=v_K(0)-i\frac{\varepsilon^3}{(2\pi)^4}\sum_{K=K_1-K_2+K_3}\eta(K_1)\overline{\eta(K_2)}\eta(K_3)\frac{1-e^{-it\Delta\omega_{KK_1K_2K_3}}}{i\Delta\omega_{KK_1K_2K_3}}+r_K(t)
\end{equation*}
with $K_1,K_2,K_3\in\Z_L^2$ and $r_K(t)$ a small remainder. As $L$ goes to infinity, the lattice $\Z_L^2$ becomes more and more refined and this is a convergent Riemann sum, a priori only for $t\ll L$ due to the oscillating term $e^{it\Delta\omega}$. In this case, it is equivalent to
\begin{equation*}
L^4\int_{K=k_1-k_2+k_3}\eta(k_1)\overline{\eta(k_2)}\eta(k_3)\frac{1-e^{-it\Delta\omega_{kk_1k_2k_3}}}{i\Delta\omega_{kk_1k_2k_3}}\drm k_1\drm k_2\drm k_3=L^4\int_\R\frac{1-e^{-it\xi}}{i\xi}\widehat R_K(\xi)\dd\xi
\end{equation*}
Finally, this localizes on the resonant manifold for large time due to the oscillating factor with
\begin{equation*}
\lim_{t\to\infty}\int_\R\frac{1-e^{-it\xi}}{i\xi}\widehat R_K(\xi)\dd\xi=\pi \widehat R_K(0)+\int_\R\frac{\widehat R_K(\xi)-\widehat R_K(0)}{i\xi}\dd\xi.
\end{equation*}
For $L\le t\le L^{d-\delta}$ in dimension $d\ge3$, the convergence follows from number theoretic results and was obtained by Buckmaster, Germain, Hani and Shatah in \cite{BGHS}. To the best of our knowledge, this is still an open question for $d=2$. For any time $t\in\R$, the nonresonant sum is bounded by
\begin{equation*}
\left|\sum_{\substack{K=K_1-K_2+K_3\\\Delta\omega_{KK_1K_2K_3}\neq 0}}\eta(K_1)\overline{\eta(K_2)}\eta(K_3)\frac{1-e^{-it\Delta\omega_{KK_1K_2K_3}}}{i\Delta\omega_{KK_1K_2K_3}}\right|\le CL^{4+\delta}
\end{equation*}
for any $\delta>0$ as proved by Faou, Germain and Hani \cite{fgh} while the resonant sum converges to $\CR$ with
\begin{equation*}
\sum_{\substack{K=K_1-K_2+K_3\\\Delta\omega_{KK_1K_2K_3}=0}}\eta(K_1)\overline{\eta(K_2)}\eta(K_3)\frac{1-e^{-it\Delta\omega_{KK_1K_2K_3}}}{i\Delta\omega_{KK_1K_2K_3}}\simeq \frac{2tL^2\log(L)}{\zeta(2)}\CR_K(\eta)
\end{equation*}
which dominates as soon as $t\ge L^{2+\delta}$. Using a naive bound instead of their result, one has
\begin{equation*}
\left|\sum_{\substack{K=K_1-K_2+K_3\\\Delta\omega_{KK_1K_2K_3}\neq 0}}\eta(K_1)\overline{\eta(K_2)}\eta(K_3)\frac{1-e^{-it\Delta\omega_{KK_1K_2K_3}}}{i\Delta\omega_{KK_1K_2K_3}}\right|\le CL^5
\end{equation*}
and hence, that the resonant sum dominates for $t\ge L^3$. In the end, the quasi-resonances dominate for times $1\ll t\ll L^2$ where one obtains first a continuous limit which then a localization on the resonant manifold. For times $L^2\ll t\ll\frac{1}{h}$, the resonances dominate with first a localization on the discrete resonant manifold which then converges to the $\CR$ operator. The parameter $h$ corresponds to the spatial truncation and can be taken arbitrary small.

\medskip

In order to recover the wave kinetic operator, we then make the assumption of random phase (RP). We consider the randomization of the initial data
\begin{equation*}
\varphi_\theta(x)=\frac{1}{(2\pi)^2}e^{-\frac{1}{2}h^2|x|^2}\sum_{K\in\Z_L^2}\eta_Ke^{i\theta_K}e^{iK\cdot x}
\end{equation*}
where $(\theta_K)_{K\in\Z_L^2}$ are independent and identically distributed uniform random variables in $[0,2\pi]$. One could also consider the random phase and amplitude assumption (RPA) with Gaussian random variables instead of uniform random variables on the circle. Recall that
\begin{equation*}
\CT_k(n)=\int_{\substack{k=k_1-k_2+k_3\\\Delta\omega_{kk_1k_2k_3}=0}}n(k)n(k_1)n(k_2)n(k_3)\left(\frac{1}{n(k)}-\frac{1}{n(k_1)}+\frac{1}{n(k_2)}-\frac{1}{n(k_3)}\right)\drm k_1\drm k_2\drm k_3
\end{equation*}
for $k\in\R^2$.

\begin{theorem}\label{TheoremRandom}
Let $(h,L,\sigma)$ be in the asymptotic scaling \eqref{regime1}, $K\in\Z_L^2$, $u$ the solution to \eqref{NLS} with random initial data $u(0)=\varepsilon\varphi_\theta$, $\delta>0$ and assume $\varepsilon\ll\frac{h^2}{L}$. Then $v(t)=e^{-it\Delta}u(t)$ satisfies
\begin{equation*}
\IE\big[|\langle v(t)\rangle_{K,\sigma}|^2\big]=\frac{\sigma^4\varepsilon^2}{(\sigma^2+h^2)^2}|\eta(K)|^2+\varepsilon^4E_1(t,\eta)+\frac{t\varepsilon^6L^4}{(2\pi)^8}\CT_K(\eta)+o(t\varepsilon^5L^4)
\end{equation*}
for $L^\delta\le t\le L^{1-\delta}$. For $L^{2+\delta}\le t\le\frac{1}{hL^\delta}$, we have 
\begin{equation*}
\IE\big[|\langle v(t)\rangle_{K,\sigma}|^2\big]=\frac{\sigma^4\varepsilon^2}{(\sigma^2+h^2)^2}|\eta(K)|^2+\varepsilon^4E_1(t,\eta)+\frac{2t^2\varepsilon^6L^2\log(L)}{\zeta(2)(2\pi)^8}\CT_K(\eta)+o\big(t^2\varepsilon^5L^2\log(L)\big).
\end{equation*}
Moreover, we have
\begin{equation*}
\IE\Big[\langle v(t)\rangle_{K,\sigma}\overline{\langle v(t)\rangle_{K',\sigma}}\Big]=o\big(\varepsilon^4L^4+t\varepsilon^4L^2\log(L)\big)
\end{equation*}
for $K\neq K'$.
\end{theorem}

The almost sure expansion
\begin{equation*}
\langle v(t)\rangle_{K,\sigma}=\frac{\sigma^2\varepsilon}{\sigma^2+h^2}\eta(K)e^{it\theta_K}-i\varepsilon^3V_K^1(t)-\varepsilon^5V_K^2(t)+\mathcal{O}(\varepsilon^7)
\end{equation*}
gives
\begin{align*}
|\langle v(t)\rangle_{K,\sigma}|^2&=\frac{\sigma^4\varepsilon^2}{(\sigma^2+h^2)^2}|\eta(K)|+\frac{2\sigma^2\varepsilon^4}{\sigma^2+h^2}\Im\big(\overline{\eta(K)}V_K^1(t)e^{-it\theta_K}\big)\\
&\quad+\varepsilon^6\Big(|V_K^1(t)|^2+\frac{2\sigma^2}{\sigma^2+h^2}\Im\big(\overline{\eta(K)}V_K^2(t)e^{-it\theta_K}\big)\Big)+\mathcal{O}(\varepsilon^8).
\end{align*}
The term $E_1(t,\eta)$ can be seen as the effect of a {\em normal form} transformation. In general, it still dominates the second order term in the expansion with respect to $\varepsilon$. However since the nonlinearity satisfies $g(\overline u)=\overline{g(u)}$, we have $E_1(t,\eta)=-E_1(-t,\eta)$ and
\begin{equation*}
\IE\big[|\langle v(t)\rangle_{K,\sigma}|^2\big]+\IE\big[|\langle v(-t)\rangle_{K,\sigma}|^2\big]=2\frac{\sigma^4\varepsilon^2}{(\sigma^2+h^2)}|\eta(K)|^2+2E_2(t,\eta)+\mathcal{O}(\varepsilon^8).
\end{equation*}
The second theorem then follows with the same kind of limit as for the first theorem. Similar computations yield the propagation of chaos.

\begin{remark}
Our result concern solutions to \eqref{NLS} with small initial data. Considering $U(t)=\frac{1}{\varepsilon L^2}u(t)$ gives a solution to
\begin{equation*}
i\partial_tU=\Delta U+\varepsilon^2L^4|U|^2U
\end{equation*}
with initial data 
\begin{equation*}
U(0)=\frac{1}{L^2}\varphi=e^{-\frac{1}{2}h^2|x|^2}\frac{1}{(2\pi L)^2}\sum_{K\in\Z_L^2}\eta_Ke^{iKx}
\end{equation*}
which is the localization of a periodic function normalized in $L^2(\IT_L^2)$. Our result can be stated as an asymptotic for the spatial localization of initial data normalized in $L^2(\IT_L^2)$ for
\begin{equation*}
i\partial_tu=\Delta u+\lambda|u|^2u
\end{equation*}
with $\lambda=\varepsilon^2L^4\ll\frac{h^4}{L^2}L^4=h^4L^2$. The scaling $U(t)=\frac{h}{\varepsilon L}u(t)$ gives an initial data normalized in $L^2(\R^2)$ and a nonlinearity $\lambda=\frac{\varepsilon^2 L^2}{h^2}\ll h^2$.
\end{remark}

\medskip

In Section \ref{SectionScatt}, we illustrate the relation between scattering and the resonant manifold. We also gives the expansion that will be used here. In Section \ref{SectionInit}, we introduce our initial data and estimates its size. In Section \ref{SectionDeter}, we prove the asymptotic for deterministic initial data, that is theorem \ref{TheoremDeter}. In Section \ref{SectionRandom}, we prove the asymptotic for random initial data, that is theorem \ref{TheoremRandom}.

\medskip

We define the Fourier transformation
\begin{equation*}
f_k \equiv \widehat f(k) =  \int_{\mathbb{R}^2} f(x) e^{- i k \cdot x} \dd x, \quad f(x) = \frac{1}{(2\pi)^2} \int_{\mathbb{R}^2} f_k e^{i k \cdot x} \dd k
\end{equation*}
and we have the Plancherel idendity 
\begin{equation*}
\int_{\mathbb{R}^2} f(x) g(x) \dd x = \frac{1}{(2\pi)^2}  \int_{\mathbb{R}^2} \widehat f(k) \widehat g(k) \dd k. 
\end{equation*}

\medskip
\noindent {\bf
Acknowledgement.}
Both authors are supported by a
Simons Collaboration Grant on Wave Turbulence. The authors would also like to thank Sergei Kuksin for helpful discussions on the topic. 

\section{Scattering and resonant manifold}\label{SectionScatt}


In this section, we illustrate the relation between decay estimates piloting the scattering and the regularity of the resonant manifold. 
Let $(k,k_1,k_2,k_3) \in (\R^2)^4$ with $k=k_1-k_2+k_3$ and consider the change of variable $k_1=k+a$, $k_3=k+b$ which yields $k_2=k+a+b$. In this case, we have 
\begin{equation*}
\Delta\omega_{kk_1k_2k_3} = |k|^2 + |k+a+b|^2 - |k+a|^2 - |k +b|^2 = 2 a \cdot b
\end{equation*}
thus the operator $\eqref{trilin}$ writes
\begin{equation*}
R_k(t,u,v,w)  = \int_{(\R^2)^2} e^{i 2 t a \cdot b} u_{k + a}\widebar v_{k+b } w_{k + a + b} \dd a \dd b.  
\end{equation*}
If we define $H(a,b) = 2 a\cdot b$, then the co-area formula gives
\begin{equation*}
R_k(t,u,v,w)  = \int_{\R} e^{i  t \xi} \left(\int_{S_k(\xi)}u_{k + a}\widebar v_{k+b } w_{k + a + b}\dd S_k(\xi) \right) \dd \xi
\end{equation*}
with the microcanonical measure
\begin{equation*}
\dd S_k(\xi)=\frac{\dd a \dd b|_{S_k(\xi)}}{\Norm{\nabla H(a,b)}{}}, 
\end{equation*}
where $\dd a \dd b|_{S_k(\xi)}$ denotes the measure induced by the Euclidean measure on the manifold $S_k(\xi)$ embedded into $\R^4$, and $\Norm{\cdot}{}$ the Euclidean norm. The following proposition relates $R$ to the CR operator \eqref{CR}, justifying the notation 
\begin{equation*}
\CR_k(u,v,w)=\int_{\substack{k=k_1-k_2+k_3\\|k|^2=|k_1|^2-|k_2|^2+|k_3|^2}}u(k_1)\overline{v(k_2)}w(k_3)\dd k_1\dd k_2\dd k_3.
\end{equation*}

\begin{proposition}\label{LemCR}
For $k\in\R^2$, we have
$$
\int_{S_{k}(0)} 
u_{k_1} \widebar v_{k_2}  w_{k_3} \, \dd S_{k}(0)  = \frac12 \int_{\R} \int_{\R^2}
u_{k + a} \widebar v_{k + \lambda a^\perp}  w_{k + a + \lambda a^\perp} \dd \lambda \dd a
= \mathcal{R}_k(u,v,w).
$$
\end{proposition}

\begin{proof}
The set $S_k(0)$ is parametrized by $(a,b) = (a,\lambda a^\perp)$ for $\lambda\in\R$ with $(a_1,a_2)^\perp:=(a_2,-a_1)$ and the microcanonical measure is given by
\begin{equation*}
\drm S_k(0)=\frac{\drm a\drm(\lambda a^\perp)}{2|a|\sqrt{1+\lambda^2}}.
\end{equation*}
We have to calculate the Jacobian of the application $\Psi(a,\lambda):=(a,\lambda a^\perp)\in\R^4$ for $(a,\lambda)\in\R^2\times\R$. We have 
\begin{align*}
\partial_{a_1} \Psi &= (1,0,0,-\lambda),\\
\partial_{a_2} \Psi &= (0,1,\lambda,0),\\
\partial_{\lambda} \Psi &= (0,0,a_1,a_2), 
\end{align*}
hence the metric matrix induced by the Euclidian space is 
\begin{equation*}
g
\begin{pmatrix}
1 + \lambda^2 & 0 & - \lambda a_2 \\
0 & 1 + \lambda^2 & \lambda a_1 \\
- \lambda a_2 &\lambda a_1 &|a|^2 \\
\end{pmatrix}
\end{equation*}
thus the volume form is 
\begin{equation*}
\sqrt{\det g}\ \dd a \dd \lambda =\sqrt{(1 + \lambda^2)^2 |a|^2 -(1 + \lambda^2)\lambda^2(a_1^2 + a_2)^2} \dd a \dd \lambda= |a| \sqrt{( 1+ \lambda^2)} \dd a \dd \lambda.
\end{equation*}
This shows that 
\begin{equation*}
\int_{S_{k}(0)} 
u_\ell \widebar v_m w_j \, \dd S_{k}(0) = \frac{1}{2} \int_{\R} \int_{\R^2}
u_{k + a} \widebar v_{k + \lambda a^\perp} w_{k + a + \lambda a^\perp} \dd \lambda \dd a.
\end{equation*}
and completes the proof.
\end{proof}

We now relate the regularity of the resonant manifold to the dispersive effects of the equations. The following result if of course far from being optimal however its proofs is somehow general and could be adapted to different frameworks. For the Schrödinger equation, optimality in terms of time and space dependence is given by Strichartz estimates as we will explain below. Let $X^{\beta,r}$ be the Banach space equipped with the norm
\begin{equation*}
\Norm{u}{X^{\beta,r}} = \sup_{\substack{\alpha = (\alpha_1,\alpha_2) \in \N^2 \\|\alpha_1| + |\alpha_2| \leq r}}\sup_{k \in \R^2} \langle k \rangle^\beta |\partial_{k}^\alpha \varphi_k |.
\end{equation*}
In particular, these spaces are natural in the context of wave turbulence, see for example \cite{fgh}.
The following proposition shows that for smooth functions $u,v,w$, the resonant manifold are smooth in $\xi$ in a weak sense.  

\begin{proposition}
Let $\beta > 2$ and $u,v,w\in X^{\beta,2}$. We have
\begin{equation*}
\Norm{R(t,u,v,w)}{X^{\beta,0}}\lesssim\frac{\langle\log t\rangle^2}{\langle t\rangle^2}\Norm{u}{X^{\beta,2}}\Norm{v}{X^{\beta,2}}\Norm{w}{X^{\beta,2}}
\end{equation*} 
for $t\in\R$. In particular, the application 
\begin{equation*}
\xi \mapsto  \int_{S_{k}(\xi)} u_\ell \widebar v_m w_j\dd S_{k}(\xi)
\end{equation*}
from $\R$ to $\C$ is of class $\mathcal{C}^\alpha$ for any $\alpha\in(0,1)$.
\end{proposition}

\begin{proof}
We have 
\begin{equation*}
\int_{S_{k}(\xi)} 
u_\ell \widebar v_m w_j \, \dd S_{k}(\xi) = \frac{1}{2\pi} \widehat R_k (\xi,u,v,w)
=  \frac{1}{2\pi} \int_{\R^2} e^{- i \xi t } R_k( t, u,v,w) \dd t
\end{equation*}
hence the regularity result will indeed follows from the bound
\begin{equation*}
\Norm{R(t,u,v,w)}{X^{\beta,0}}\le C\frac{\langle\log t\rangle^2}{\langle t\rangle^2}\Norm{u}{X^{\beta,2}}\Norm{v}{X^{\beta,2}}\Norm{w}{X^{\beta,2}}
\end{equation*}
for $t\in\R$. Let $(k,k_1,k_2,k_3) \in (\R^2)^4$ with $k=k_1-k_2+k_3$. Again, the change of variable $k_1=k+a$, $k_3=k+b$ yields $k_2=k+a+b$ and we have 
\begin{equation*}
\Delta\omega_{kk_1k_2k_3} = |k|^2 + |k+a+b|^2 - |k+a|^2 - |k +b|^2 = 2 a \cdot b
\end{equation*}
thus giving
\begin{equation*}
R_k(t,u,v,w)  = \int_{\R^2 \times \R^2} e^{2it a \cdot b} u_{k + a} \widebar v_{k+b } w_{k + a + b} \dd a \dd b.  
\end{equation*}
Let $\chi: \R^2 \to \R$ such that $\chi(z) \equiv 1$ for $\min(|z_2|,|z_2|) \leq  1$, $\chi(z) \equiv 0$ for $\min(|z_2|,|z_2|)  > 2$ and $\chi \in [0,1]$. We define 
\begin{equation*}
r_k(\delta,u,v,w) := \int_{\R^2 \times \R^2} e^{2i t a \cdot b} u_{k + a} \widebar v_{k+b } w_{k + a + b} \chi\Big( \frac{b}{\delta}\Big) \dd a \dd b.  
\end{equation*}
We first prove that for $\beta>2$, we have
\begin{equation}\tag{$\star$}
\label{brahms2}
\Norm{\langle k \rangle^\beta r_k(\delta,u,v,w) }{L^\infty}  \leq C_\beta \delta \Norm{\langle k \rangle^\beta  u(k)}{L^\infty}\Norm{\langle k \rangle^\beta v(k)}{L^\infty}\Norm{\langle k \rangle^\beta  w(k)}{L^\infty}
\end{equation}
for a constant $C_\beta$ uniform with respect to $\delta$.

\medskip

\noindent\underline{\it Proof of \eqref{brahms2}.} We have 
\begin{align*}
\Norm{\langle k \rangle^\beta r_k(\delta)}{L^\infty} &\leq C_s  \Norm{\langle k \rangle^\beta  u(k)}{L^\infty}\Norm{\langle k \rangle^\beta v(k)}{L^\infty}\Norm{\langle k \rangle^\beta  w(k)}{L^\infty}\nonumber \\
&\quad\times \left( 
\int_{\R^2\times\R^2}\frac{\langle k \rangle^\beta}{\langle k +a  \rangle^\beta\langle k +b  \rangle^\beta\langle k +a + b \rangle^\beta}\chi\Big( \frac{b}{\delta}\Big) \dd a \dd b \right)
\end{align*}
and using the fact that 
\begin{equation*}
\langle k \rangle^\beta \leq C_\beta \big( \langle k + a \rangle^\beta+ \langle k + b \rangle^\beta + \langle k +a +b \rangle^\beta\big),
\end{equation*}
the integral to be estimated can be divided into three terms
\begin{equation*}
\int_{\R^2\times\R^2}\left(\frac{1}{\langle k +b  \rangle^\beta\langle k +a + b \rangle^\beta}+\frac{1}{\langle k +a  \rangle^\beta \langle k +a + b \rangle^\beta}+\frac{1}{\langle k +a  \rangle^\beta \langle k +b \rangle^\beta }\right)\chi\Big( \frac{b}{\delta}\Big) \dd a \dd b.
\end{equation*}
As $\beta>1$, the first term can be bounded by 
\begin{align*}
\int_{\R^2}\frac{1}{\langle k +b\rangle^\beta} \chi\Big(\frac{b}{\delta}\Big)  \dd b &
\leq \int_{\R} \left(\int_{-2\delta}^{2\delta} \frac{1}{ ( 1 + |k_1 + b_1|^2 + |k_2 + b_2|^2)^{\frac{\beta}{2}}} \dd b_1\right)\dd b_2\\
&\quad+\int_{\R} \left(\int_{-2\delta}^{2\delta} \frac{1}{ ( 1 + |k_1 + b_1|^2 + |k_2 + b_2|^2)^{\frac{\beta}{2}}}  \dd b_2\right)\dd b_1\\
& \leq C\delta
\end{align*}
as well as
\begin{equation*}
\int_{\R^2}\frac{1}{\langle k+a+b\rangle^\beta}\dd a=\int_{\R^2}\frac{1}{\langle a\rangle^\beta}\dd a<\infty
\end{equation*}
since $\beta>2$. To bound the second term, we note that 
\begin{equation*}
\langle b \rangle \leq \langle k + a + b\rangle  \langle k + a \rangle
\end{equation*}
and thus for $\alpha > 1$ and $\beta - \alpha > 1$, 
\begin{align*}
\int_{\R^2\times\R^2}\frac{1}{\langle k +a  \rangle^\beta \langle k +a + b \rangle^\beta}\chi\Big( \frac{b}{\delta}\Big) \dd a \dd b & \leq \int_{\R^2\times\R^2}\frac{1}{\langle k +a  \rangle^{\beta - \alpha} \langle k +a + b \rangle^{\beta- \alpha} } \frac{1}{\langle b \rangle^\alpha}\chi\Big( \frac{b}{\delta}\Big) \dd a \dd b\\
& \leq \int_{\R^2} \frac{1}{\langle b \rangle^\alpha}\chi\Big( \frac{b}{\delta}\Big)  \dd b \\
&\leq C \delta, 
\end{align*}
such $\alpha$ exists since $\beta>2$. And finally, as $\beta > 2$, the last term can be bounded by 
\begin{equation*}
\int_{\R^2\times\R^2}\frac{1}{\langle k +a  \rangle^\beta \langle k +b \rangle^\beta }\chi\Big( \frac{b}{\delta}\Big) \dd a \dd b  \leq C\int_{\R^2}\frac{1}{\langle k +b \rangle^\beta }\chi\Big( \frac{b}{\delta}\Big)  \dd b \leq C \delta
\end{equation*}
and this proves \eqref{brahms2}. 

\medskip

In order to bound $R_{k}(t,u,v,w)$, we want to integrate by part the oscillatory term and gain powers of $t$, up to a lost of derivatives. We thus have to be sure that we can distribute a number of derivatives where we want. We can write formally after integrations by part
\begin{equation*}
\int_{(\R^2)^2} e^{i 2 t a \cdot b} u_{k + a} \widebar  v_{k+b }  w_{k + a + b} \dd a \dd b = \frac{1}{(2it)^2 }\int_{(\R^2)^2} e^{i 2 t a \cdot b}  \frac{1}{b_1 b_2} \partial_{a_1} \partial_{a_2}(u_{k + a} \widebar  v_{k+b }  w_{k + a + b}) \dd a \dd b.  
\end{equation*}
Essentially, under some decay assumptions on $u$, $v$ and $w$, and their derivatives, this term will be of order $\langle t \rangle^{-2}$ up to logarithmic singularities due to the singularities in $a_1 = 0$ and $a_2 = 0$. To make the estimate rigorous, we write 
\begin{align*}
\int_{\R^2 \times \R^2} e^{i 2 t a \cdot b} u_{k + a} \widebar v_{k+b }  w_{k + a + b} \dd a \dd b&=  \int_{\R^2 \times \R^2} e^{i 2 t a \cdot b} u_{k + a} \widebar v_{k+b }  w_{k + a + b}\chi\Big(\frac{b}{\delta}\Big) \dd a \dd b\\
&\quad+ \int_{\R^2 \times \R^2} e^{i 2 t a \cdot b}u_{k + a} \widebar v_{k+b }  w_{k + a + b} \Big( 1 - \chi\big(\frac{b}{\delta}\big) \Big) \dd a \dd b\\
&=: r_k(t,\delta) + R_k(t,\delta).
\end{align*}
The first term is bounded with \eqref{brahms2}, that is
\begin{equation*}
\Norm{\langle k \rangle^\beta r_k(t,\delta)  }{L^\infty} \leq C \delta \Norm{u}{X^{\beta,2}} \Norm{v}{X^{\beta,2}}\Norm{w}{X^{\beta,2}}. 
\end{equation*}
To estimate the second term, we perform the integration by part, and obtain 
\begin{equation*}
\frac{1}{(2it)^2 } \int_{\R^2 \times \R^2} e^{i 2 t a \cdot b} \frac{1}{b_1 b_2}  (\partial_a^{\alpha_u}u_{k + a}) \widebar v_{k+b } (\partial_a^{\alpha_w}  w_{k + a + b}) \Big( 1 - \chi\big(\frac{b}{\delta}\big)\Big) \dd a \dd b
\end{equation*}
with $|\alpha_u|+|\alpha_w|= 2$. This term can be estimated as for \eqref{brahms2} using
\begin{equation*}
|\partial_k^{\alpha_u}u_{k}| \leq \frac{1}{\langle k \rangle^{\beta}}\Norm{u}{X^{\beta,2}}\quad\mbox{and}\quad |\partial_k^{\alpha_w}w_{k}| \leq \frac{1}{\langle k \rangle^{\beta}}\Norm{w}{X^{\beta,2}},
\end{equation*}
except that the measure $\chi(\frac{b}{\delta})$ is replaced by the measure $\frac{1}{b_1 b_2} (1 - \chi(\frac{b}{\delta}))$. For example, we have
\begin{align*}
\left|\int_{\R^2}\frac{1}{\langle k +b\rangle^s} \frac{1}{b_1 b_2} \Big(1 - \chi\big(\frac{b}{\delta}\big) \Big)  \dd b\right|&\leq \int_{|b_1|>\delta}\int_{|b_2|>\delta} \frac{1}{|b_1| |b_2|( 1 + |k_1 + b_1|^2 + |k_2 + b_2|^2)^{\frac{\beta}{2}}} \dd b_1 \dd b_2\\
&\leq  C\langle\log \delta\rangle^2
\end{align*}
thus we obtain 
\begin{equation*}
\Norm{\langle k \rangle^\beta R_k(t,\delta)  }{L^\infty} \leq  \frac{C}{\langle t \rangle} \langle \log \delta\rangle \Norm{u}{X^{\beta,2}} \Norm{v}{X^{\beta,2}}\Norm{w}{X^{\beta,2}}. 
\end{equation*}
Taking $\delta = \frac{1}{\langle t\rangle}$ gives
\begin{equation*}
\Norm{\langle k \rangle^\beta ( r_k(t,\delta) + R_k(t,u,v,w))}{L^\infty}\le C\frac{\langle\log t\rangle^2}{\langle t\rangle^2}\Norm{u}{X^{\beta,2}}\Norm{v}{X^{\beta,2}}\Norm{w}{X^{\beta,2}}
\end{equation*}
and completes the proof. Indeed, the regularity $\mathcal{C}^\alpha$ for $\alpha<1$ comes for example from the Besov spaces, see Section $2.7$ in \cite{BCD}.
\end{proof}

This kind of result can be used as starting point to prove that the solution of \eqref{NLS} scatters at $t = \pm \infty$ for smooth initial data, that is $v(t)=e^{-it\Delta}u(t)$ with $u$ the solution of equation \eqref{NLS} has  limits $v_{\pm \infty}$ when $t \to \pm \infty$. Of course, it is know that the solution scatters under the general condition that $u(0) = v(0)$ is in the space 
$$
\Sigma = \{ \varphi(x) \in H^1(\R),\quad  |x| \varphi(x) \in L^2(\R^2)\}, 
$$
or equivalently that $|k| \varphi_k \in L^2(\R^2)$ and $\partial_k \varphi_k \in  L^2(\R^2)$. The result holds without restriction for any $\varphi \in \Sigma$, and for $\Norm{\varphi}{\Sigma}$ small enough in the focusing case. The main argument for \eqref{NLS} lies on a pseudoconformal conservation law argument. As a consequence, we can study the application 
$$
u(0)  = \varphi\mapsto U(t,\varphi) = v(t)
$$
which is well defined for $t \in \R$ and for $\varphi \in \Sigma$. In the following, we will consider the case where $\varphi$ is small in $\Sigma$. In particular, the sign in front of the non linearity is not important here. Using the result of Carles and Gallagher \cite{CaGa09}, the application $U(t,\varphi)$ is analytic, and by expanding this operator near zero, we obtain the following result.

\begin{theorem}[\normalfont Carles \& Gallagher \cite{CaGa09}] \label{TheoremCaGa}
There exists $\varepsilon_0>0$ such that for $\varphi \in \Sigma$ with $\Norm{\varphi}{\Sigma} \leq \varepsilon_0$, there exists a global solution  $v(t) =  U(t,\varphi)$ to the equation
\begin{equation*}
v_k(t) = \varphi_k -i  \int_{0}^t R_k\big(t,v(s),v(s),v(s)\big) \dd s, 
\end{equation*}
for which we have 
\begin{equation*}
\Norm{v(t)}{\Sigma}\leq 2 \Norm{\varphi}{\Sigma}
\end{equation*}
for all $t\in\R$ and the analytic expansion in $\Sigma$
\begin{equation*}
v_k(t)  = \varphi_k + \sum_{n \geq 3} (-i)^nV_k^n(t)
\end{equation*}
where 
\begin{equation*}
\forall\, n \in \N,  \forall\, t \in \R, \quad  \Norm{V^n(t)}{\Sigma} \leq C  \Norm{\varphi}{\Sigma}^{2n + 1}. 
\end{equation*}
Finally, we have 
\begin{align*}
V_k^1(t) = &\int_0^t R_k(s, \varphi,\varphi,\varphi) \dd s,\\
V_k^2(t) = &2\int_0^t \int_{0}^{s} R_k(s, \varphi,\varphi, R(s',\varphi,\varphi,\varphi)) \dd s'\dd s +\int_0^t \int_{0}^{s} R_k(s, \varphi, R(s',\varphi,\varphi,\varphi),\varphi) \dd s'\dd s.
\end{align*}
\end{theorem}

As explained in the introduction, the proof of this result is based on Strichartz estimates, and conservation laws of the Schr\"odinger equation and this express the link between the regularity of the resonant manifolds and the scattering effect. Note that the result holds in fact not only near $0$, but around any solution of \eqref{NLS}.

\section{Localization of periodic initial data}\label{SectionInit}



For $\epsilon>0$ and $K\in\IR^2$, consider the functions
\begin{equation*}
g_{K,\epsilon} (x) := \frac{1}{(2\pi)^2}e^{i K \cdot x } e^{ - \frac{1}{2}\epsilon^2 |x|^2}\quad\text{and}\quad\widehat g_{K,\epsilon}(k) :=  \frac{1}{2 \pi \epsilon^2}e^{- \frac{1}{2\epsilon^2}|k - K|^2}
\end{equation*}
with $x,k\in\R^2$, see Lemma \ref{fouriertransformgaussians} for Gaussian calculations. With this normalization, $\widehat g_{K,\epsilon}$ is of integral one and converges to the Dirac distribution $\delta_0(\cdot-K)$ as $\epsilon$ goes to $0$. For $L>0$, set
\begin{equation*}
\Z_L^2 = \left\{ \left(\frac{n_1}{L}, \frac{n_2}{L}\right)\ ;\ n_1,n_2\in\Z\right\}
\end{equation*}
and let $\eta_K$ be the trace on $\Z_L^2$ of a smooth function $\eta: \R^2 \mapsto \C$. The reader should have in mind that $\epsilon$ will be small while $L$ will be large. For simplicity, we assume that $\eta$ is compactly supported on a domain $B \subset \R^2$ independent of $L$, this condition could be easily relaxed to some decay assumption with respect to $\langle K \rangle$ of Sobolev type. We consider the initial data
\begin{equation*}
\label{init}
\varphi(x):= \sum_{K\in\Z_L^2}\eta_K g_{K,h}(x) = \frac{1}{(2\pi)^2} e^{- \frac{1}{2} h^2 |x|^2}\sum_{K\in\Z_L^2}\eta_K e^{i K \cdot x}=: e^{- \frac{1}{2} h^2 |x|^2} F_L(x) \nonumber
\end{equation*}
which is essentially a periodic function $F_L$ with large period $2 \pi L$ embedded in $\Sigma$ by Gaussian truncation. Since $\eta$ is compactly supported, the set of indices $K\in B\cap\Z_L^2$ is bounded and
\begin{equation*}
\Norm{\partial_x^\alpha F_L}{L^\infty} \leq C_{\alpha}L^2 
\end{equation*}
for some constant $C_\alpha>0$ independent of $L$ and all $\alpha\in\N^2$. In frequency, this writes
\begin{equation*}
\varphi_k=\sum_{K\in\Z_L^2}\frac{\eta_K}{2\pi h^2}e^{-\frac{1}{2h^2}|k-K|}
\end{equation*}
for $k\in\R^2$. As $h$ goes to $0$, this converges to
\begin{equation*}
\sum_{K\in\Z_L^2}\eta_K\delta_0(k-K),
\end{equation*}
that is the Fourier transform of $F_L$ as a function on $\IR^2$ which is not in $\Sigma$ since it does not decrease at infinity. As $L$ goes to infinity, the lattice $\Z_L^2$ becomes more and more refined and converges to $\R^2$. To deal with almost Dirac functions and the limit $L \to \infty$, we introduce another scale of observation $\sigma>0$. For a function $v$ which is expected to be close to a $2 \pi L$-periodic function, we define the {\em coarse grained} quantity in frequency
\begin{equation*}
\langle v\rangle_{K,\sigma} :=   \int_{\R^2}e^{- \frac{1}{2 \sigma^2} |k- K|^2} \widehat v(k) \dd k = (2 \pi)^3 \sigma^2  \int_{\R^2}\overline{g_{K,\sigma}(x)}v(x) \dd x
\end{equation*}
for all $K\in\Z_L^2$. The initial data $\varphi$ is a sum of $g_{K,h}$ which converges in frequency to Dirac distributions as $h$ goes to $0$. In other works such as \cite{DH1,fgh}, the quantity that is controlled is the $L^\infty(\Z_L^2)$ distance between the discrete Fourier coefficient and the continuous limit. This coarse grained quantity is here a natural object for dealing with functions defined on a set with continuous spectrum. Note that in the end, this is equivalent since the test function is equal to $1$ at $K\in\Z_L^2$ while exponentially small at the other site of the lattice for $\sigma L\ll 1$. Indeed for fixed $\sigma$ and $L$ with $h$ going to $0$, we have
\begin{equation*}
\int_{\R^2}e^{- \frac{1}{2 \sigma^2} |k- K|^2} \widehat g_{K_1,h}(k)  \dd k \simeq e^{- \frac{1}{2 \sigma^2} |K_1- K|^2} = \left|
\begin{array}{l}
1 \quad \mbox{if}\quad K=K_1,\\[1ex]
\mathcal{O}( e^{- \frac{1}{2 (\sigma L)^2}})\quad\mbox{if}\quad K\neq K_1. 
\end{array}
\right.
\end{equation*}
Hence for $h\ll\sigma\ll\frac{1}{L}$, we get
\begin{equation*}
\langle\varphi\rangle_{K,\sigma}\simeq\eta(K)
\end{equation*}
in the limit and these coarse grained quantities will be the natural ones to study to deal both with the Dirac limit and the large period limit. The following proposition makes this statement more precise.

\begin{proposition}
There exists a constant $C>0$ such that
\begin{equation*}
\sup_{K\in\Z_L^2}\big|\langle\varphi\rangle_{K,\sigma}-\frac{\sigma^2}{\sigma^2+h^2}\eta(K)\big|\leq C\Norm{\eta}{L^\infty}L^2 e^{-\frac{1}{2 L^2 ( h^2 + \sigma^2)}},
\end{equation*}
which implies in particular
\begin{equation*}
\sup_{K\in\Z_L^2}\big|\langle\varphi\rangle_{K,\sigma}-\eta(K)\big|\leq C\Norm{\eta}{L^\infty}\Big(\frac{h^2}{\sigma^2+h^2} + L^2 e^{-\frac{1}{2 L^2 ( h^2 + \sigma^2)}}\Big).
\end{equation*}
\end{proposition}

\begin{proof}
We have
\begin{align*}
\langle \varphi \rangle_{K,\sigma} &=  (2\pi)^3\sigma^2 \int_{\mathbb{R}^2} \overline{g_{K,\sigma}(x)} \varphi(x) \dd x \\
&= (2\pi)^3\sigma^2 \sum_{K_1} \eta_{K_1} \int_{\mathbb{R}^2}\overline{g_{K,\sigma}(x)} g_{K_1,h}(x)\dd x\\
&= (2\pi)^3\sigma^2 \sum_{K_1} \eta_{K_1} \frac{1}{(2\pi)^4}\int_{\R^2} e^{i (K_1 - K)\cdot x } e^{ - \frac{1}{2}\sigma^2 |x|^2 - \frac{1}{2} h^2 |x|^2}\dd x\\
&= \frac{\sigma^2}{2\pi}\sum_{K_1}\eta_{K_1}\frac{2\pi}{ (h^2  + \sigma^2) }e^{- \frac{1}{2 (h^2 + \sigma^2)} | K - K_1|^2}\\
&=\left( \frac{\sigma^2}{\sigma^2 + h^2} \right) \eta_K + \left( \frac{\sigma^2}{\sigma^2 + h^2} \right)\sum_{K_1 \neq K} \eta_{K_1} e^{- \frac{1}{2 ( h^2 + \sigma^2)} | K - K_1|^2}
\end{align*}
which completes the proof using
\begin{equation*}
\left|\left( \frac{\sigma^2}{\sigma^2 + h^2} \right)\sum_{K_1 \neq K} \eta_{K_1} e^{- \frac{1}{2 ( h^2 + \sigma^2)} | K - K_1|^2}\right|
\leq  C\sum_{K_1 \neq K}e^{-\frac{1}{2 L^2 ( h^2 + \sigma^2)} }.
\end{equation*}
\end{proof}

We now state the precise regime we are working with.

\begin{definition}
We say that the set of parameter $(h,L,\sigma)$ satisfy the asymptotic regime if we have $L \gg 1$,
\begin{equation}
\label{regime1}
\tag{AR} 
hL^{4+\delta_0}\le1\quad\text{and}\quad hL\ll\sigma\le h^{\frac{3}{4}}
\end{equation}
for some $\delta_0>0$.
\end{definition}

The scaling implies in particular $hL\ll1$ hence the number of period of $F_L$ observed before the space truncation is large, at least as $L^3$. The condition on $\sigma$ implies
\begin{equation*}
h\ll hL\ll\sigma\le h^{\frac{3}{4}}\le\frac{1}{L^3}\ll\frac{1}{L}\ll 1
\end{equation*}
hence the previous lemma implies that
\begin{equation*}
\sup_{K\in\Z_L^2}\big|\langle\varphi\rangle_{K,\sigma}-\eta(K)\big|\ll1.
\end{equation*}
Our asymptotic regime can be interpreted as the observation of a large number of period with an observation scale large with respect to the Dirac scale $h$ and small with respect to the scale $L$ of the lattice. Our description of the solution is limited for time $t\le\frac{1}{\sigma^2L^2}$ which is large under the asumption $\sigma\ll\frac{1}{L}$. In the asymptotic regime, we have
\begin{equation*}
\frac{1}{\sigma^2L^2}\ge\frac{1}{h^{\frac{3}{2}}L^2}=\frac{1}{h}\frac{1}{\sqrt{hL^4}}\gg\frac{1}{h}
\end{equation*}
which allows to prove our result on the time frame $t\ll\frac{1}{h}$ independent of the scale of observation. As explained in the introduction, two different phenomena will leads in the limit to the kinetic operator. First in the timeframe
\begin{equation*}
1\ll L^\delta\le t\le L^{1-\delta}\ll L^2
\end{equation*}
where the convergence is obtained by taking first continuous limit in $L$ and then a localization on the resonant manifold, whereas in the timeframe
\begin{equation*}
L^2\ll L^{2+\delta}\le t\le\frac{1}{L^\delta h}\ll\frac{1}{h}
\end{equation*}
the convergence is obtain first by a localization on the discrete resonant manifold, followed by a continuous limit $L$ using for any $\delta>0$, by using the analysis of \cite{fgh}. The parameter $h$ can be taken arbitrary small in the asymptotic regime and we have at least
\begin{equation*}
\frac{1}{h}\gg L^4\gg L^2\gg 1
\end{equation*}
hence the two types of limit can be observed.

\begin{remark}
Typically, the \eqref{regime1} is fulfilled for
\begin{equation*}
h=\frac{1}{L^{4+\alpha}}\quad\text{and}\quad\sigma=\frac{L^\beta}{L^{3+\alpha}}
\end{equation*}
with $\alpha>0$ and $0<4\beta\le\alpha$. 
\end{remark}

\medskip

In order to use the expansion from theorem \ref{TheoremCaGa}, we first estimate the norm of the initial data $\varphi$ in $\Sigma$. Then considering \eqref{NLS} with initial data $\varepsilon\varphi$ with $\varepsilon\ll\|\varphi\|_\Sigma$ allows to develop analytically the operator $U(t,\varphi)$ and this justifies the analysis of the first terms with small remainder terms. This is the content of the two following sections, first for deterministic initial data then for random initial data. 

\begin{lemma}\label{LemmaSizeIni}
In the asymptotic regime \eqref{regime1}, we have 
\begin{equation*}
\Norm{\varphi}{L^2} \leq C \Big(\frac{L\Norm{\eta}{L^\infty}}{h}\Big) \quad \mbox{and} 
\quad \Norm{\varphi}{\Sigma} \leq C \Big(\frac{L\Norm{\eta}{L^\infty}}{h^2}\Big)
\end{equation*}
for a constant $C>0$ depending only on $\eta$.
\end{lemma}

\begin{proof}
We have 
\begin{align*}
\Norm{\varphi}{L^2}^2 &= \sum_{K_1,K_2} \overline\eta_{K_1} \eta_{K_2}
\int_{\R^2} \overline{g_{K_1,h}(x)} g_{K_2,h}(x)\dd x\\
&= \frac{1}{(2\pi)^4}\sum_{K_1,K_2} \overline\eta_{K_1} \eta_{K_2} \int_{\R^2}
e^{i x \cdot (K_2 - K_1)} e^{ - h^2 |x|^2} \dd x\\
&= \frac{\pi}{h^2 (2\pi)^4}\sum_{K_1,K_2} \overline\eta_{K_1} \eta_{K_2}  e^{ - \frac{1}{4h^2}| K_1 - K_2|^2}\\
&= \frac{1}{2h^2 (2\pi)^3}\sum_{K} |\eta_{K}|^2 
+ \mathcal{O}\Big(\Norm{\eta}{L^{\infty}}^2 \frac{ L^4}{h^2}  e^{ - \frac{1}{4h^2L^2}}\Big) \\
&=  \Norm{\eta}{L^{\infty}}^2 \mathcal{O}\Big( \frac{L^2}{h^2} +  \frac{ L^4}{h^2}  e^{ - \frac{1}{4h^2L^2}}\Big).
\end{align*}
This gives the first bound by noticing that for any $N\ge0$, 
\begin{equation*}
\frac{ L^4}{h^2}  e^{ - \frac{1}{4h^2L^2}} \leq C_N\frac{L^4}{h^2} (h L)^N 
\end{equation*}
for a constant $C_N>0$, making the exponential term negligible in the asymptotic regime \eqref{regime1} by taking $N$ large enough. For the second bound, we have $\partial_{x_1} g_{K,h} = (i (K)_1 - h^2 x_1) g_{K,h}$ hence
\begin{align*}
&\Norm{\partial_{x_1}\varphi}{L^2}^2 = \sum_{K_1,K_2} \overline\eta_{K_1} \eta_{K_2}
\int_{\R^2} (-i (K_1)_1  - h^2 x_1)(i (K_2)_1  - h^2 x_1)\overline{g_{K_1,h}(x)} g_{K_2,h}(x)\dd x\\
&= \frac{1}{(2\pi)^4}\sum_{K_1,K_2} \overline\eta_{K_1} \eta_{K_2} \int_{\R^2}
 (-i (K_1)_1  - h^2 x_1)(i (K_2)_1  - h^2 x_1) e^{i x \cdot (K_2 - K_1)} e^{ - h^2 |x|^2} \dd x\\
&= \frac{1}{h^2(2\pi)^4}\sum_{K_1,K_2} \overline\eta_{K_1} \eta_{K_2} \int_{\R^2}
 (-i (K_1)_1  - h x_1)(i (K_2)_1  - h x_1) e^{\frac{i}{h }x \cdot (K_2 - K_1)} e^{ - |x|^2} \dd x\\
&=\Norm{\eta}{L^{\infty}}^2 \mathcal{O}\Big( \frac{L^2}{h^2} +  \frac{ L^4}{h^2}  e^{ - \frac{1}{4h^2L^2}}\Big)
\end{align*}
with the same arguments since $\eta$ is compactly supported, as well as the bound for $\partial_{x_2}$. We also have 
\begin{align*}
&\Norm{x_1\varphi}{L^2}^2 = \sum_{K_1,K_2} \overline\eta_{K_1} \eta_{K_2}
\int_{\R^2} x_1^2\overline{g_{K_1,h}(x)} g_{K_2,h}(x)\dd x\\
&= \frac{1}{(2\pi)^4}\sum_{K_1,K_2} \overline\eta_{K_1} \eta_{K_2} \int_{\R^2}
 x_1^2 e^{i x \cdot (K_2 - K_1)} e^{ - h^2 |x|^2} \dd x\\
&= \frac{1}{h^4(2\pi)^4}\sum_{K_1,K_2} \overline\eta_{K_1} \eta_{K_2} \int_{\R^2}
x_1^2 e^{\frac{i}{h }x \cdot (K_2 - K_1)} e^{ - |x|^2} \dd x
\end{align*}
which gives an additional factor $h^{-2}$ due to the term $x_1^2$ and completes the proof with similar arguments.
\end{proof}

\section{Deterministic initial data}\label{SectionDeter}

We compute $V_K^1(t)$ for deterministic initial data and its limit. To ensure the convergence of the developments, we work with the expansion from theorem \ref{TheoremCaGa} with initial data $\varepsilon\varphi$ with $\varepsilon\ll\|\varphi\|_\Sigma$, that is 
$\varepsilon\ll\frac{L}{h^2}$ using lemma \ref{LemmaSizeIni}. We first compute the Schr\"odinger propagation of our building blocks $g_{K,\epsilon}$. The following computation will be essential in all that follows.

\begin{lemma}\label{ComputGauss}
Let $z \in \C$ with $\Re z > 0$, $\xi = (\xi_1,\xi_2) \in \C^2$ and 
\begin{equation*}
f(x) = e^{- z |x|^2 + \xi \cdot x}
\end{equation*}
for $x\in\R^2$, with the notation $\xi \cdot x = \xi_1 x_1 + \xi_2 x_2 \in \C$.  We have 
\begin{equation*}
(e^{i t \Delta} f )(x) = \frac{1}{1 + 4izt}  e^{  - \frac{z}{1 + 4i tz}|x|^2 
+ \frac{1}{1 + 4i tz} x \cdot \xi 
+  \frac{it }{1 + 4i tz}  \xi \cdot \xi}.
\end{equation*}
\end{lemma}

\begin{proof}
The Fourier transform of the function 
\begin{align*}
f(x) =  e^{- z|x|^2 + \xi \cdot x} 
\end{align*}
is given by 
\begin{align*}
\hat f (k) &= \int_{\R^2}e^{- z|x|^2 + \xi \cdot x} e^{-ik\cdot x}\dd x=\frac{\pi}{z}e^{\frac{1}{4z}(\xi-ik)(\xi-ik)}= \frac{\pi}{z} e^{  \frac{1}{4z} ( - |k|^2  - 2 i \xi\cdot  k  + \xi\cdot \xi) }.
\end{align*}
We get
\begin{align*}
(e^{i t \Delta} f )(x) &=\frac{1}{(2\pi)^2}\int_{\R^2}e^{-it|k|^2}\widehat{f}(k)e^{ik\cdot x}\dd k\\
&=  \frac{1}{(2\pi)^2}  \frac{\pi}{z}e^{ \frac{\xi\cdot \xi}{4z}}  \int_{\R^2} e^{ - i t |k|^2 + i k\cdot x   - \frac{|k|^2}{4z}   - \frac{i}{2z}  \xi\cdot  k  } \dd k\\
&= \frac{1}{(2\pi)^2}  \frac{\pi}{z}e^{ \frac{\xi\cdot \xi}{4z}}  \int_{\R^2} e^{ - \frac{ 1 + 4z i t}{4z} |k|^2 +  k\cdot (i x     - \frac{i \xi }{2z} ) } \dd k\\
&=   \frac{1}{1 + 4izt}  e^{ \frac{\xi\cdot \xi}{4z}} e^{  - \frac{z}{1 + 4i tz} ( x - \frac{\xi}{2z}) \cdot ( x - \frac{\xi}{2z})  }\\
&=   \frac{1}{1 + 4izt}  e^{ \frac{\xi\cdot \xi}{4z}} e^{  - \frac{z}{1 + 4i tz}|x|^2 
+ \frac{1}{1 + 4i tz} x \cdot \xi 
-  \frac{1}{4z} \frac{1}{1 + 4i tz}  \xi \cdot \xi}\\
&=   \frac{1}{1 + 4izt} e^{  - \frac{z}{1 + 4i tz}|x|^2 
+ \frac{1}{1 + 4i tz} x \cdot \xi 
+  \frac{it}{1 + 4i tz}  \xi \cdot \xi}
\end{align*}
which yields the result. 
\end{proof}

Since we have 
\begin{equation*}
g_{K,\epsilon} (x) = \frac{1}{(2\pi)^2}e^{i K \cdot x } e^{ - \frac{1}{2}\epsilon^2 |x|^2},
\end{equation*}
this yields the propagation of our building blocks $g_{K,\epsilon}$ with $z = \frac{1}{2}\epsilon^2$ and $\xi = i K$. 

\begin{corollary}\label{ComputPropgk}
For $\epsilon>0$ and $K \in \Z^2_L$, we have
\begin{equation*}
(e^{i t \Delta} g_{K,\epsilon})(x) = \frac{\beta_\epsilon(t)}{(2\pi)^2}  e^{  - \frac{\epsilon^2\beta_\epsilon(t)  }{2}|x|^2   
+ i \beta_{\epsilon}(t) x \cdot K
-  i t \beta_{\epsilon}(t) |K|^2}
\end{equation*}
with 
\begin{equation*}
\beta_{\epsilon}(t):=\frac{1}{1+2it\epsilon^2}. 
\end{equation*}
\end{corollary}

In the following, we will work in the timeframe $t\epsilon^2\ll1$. Thus $\beta_\epsilon(t)\simeq 1$ and
\begin{equation*}
(e^{i t \Delta} g_{K,\epsilon})(x) \simeq \frac{1}{(2\pi)^2}  e^{  - \frac{\epsilon^2}{2}|x|^2   
+ i  x \cdot K
-  i t |K|^2}=g_{K,\epsilon}(x)e^{-it|K|^2}
\end{equation*}
hence the Schrödinger propagation acts as expected as an oscillation at speed $|K|^2$. We now compute the first order term. Recall that
\begin{equation*}
\langle V^1\rangle_{K,\sigma} = (2\pi)^3 \sigma^2 \int_0^t \int_{\mathbb{R}^2} \overline{g_{K,\sigma}}(x)e^{-i s \Delta}\big( |e^{is\Delta} \varphi|^2 e^{is\Delta} \varphi \big)(x) \dd x\dd s.
\end{equation*}
With the previous approximation, we get
\begin{equation*}
e^{-i s \Delta}\big( e^{is\Delta} g_{K_1,h} e^{-is\Delta} \overline{g_{K_2,h}} e^{is\Delta} g_{K_3,h} \big)\simeq\frac{1}{(2\pi)^4}e^{-is(|K_1-K_2+K_3|^2-|K_1|^2+|K_2|^2-|K_3|^2)}g_{K_1-K_2+K_3,h}
\end{equation*}
for $K_1,K_2,K_3\in\Z_L^2$ hence
\begin{align*}
(2\pi)^3\sigma^2\int_{\R^2}\overline{g_{K,\sigma}(x)}&e^{-i s \Delta}\big( e^{is\Delta} g_{K_1,h} e^{-is\Delta} \overline{g_{K_2,h}} e^{is\Delta} g_{K_3,h} \big)(x)\dd x\\
&\simeq\frac{1}{(2\pi)^4}e^{-is(|K|^2-|K_1|^2+|K_2|^2-|K_3|^2)}\delta_0(K-K_1+K_2-K_3).
\end{align*}
The following proposition makes this precise.

\begin{proposition}\label{PropExpansion1}
Let $(h,L,\sigma)$ in the asymptotic regime \eqref{regime1}. We have 
\begin{equation*}
\langle V^1(t)\rangle_{K,\sigma}=\frac{1}{(2\pi)^4}\sum_{K=K_1-K_2+K_3}\eta_{K_1}\overline{\eta_{K_2}}\eta_{K_3}\frac{1-e^{-it\Delta\omega_{KK_1K_2K_3}}}{i\Delta\omega_{KK_1K_2K_3}}+r_K(t,\eta) 
\end{equation*}
with the convention $\frac{1-e^{i0t}}{i0}=t$ and $r_K$ satisfying the following estimates. For $t\le\frac{1}{\sigma^2L^2}$, we have
\begin{align*}
|r_K(t,\eta)|&\le CL^6e^{-\frac{1}{2\sigma^2L^2}+C(t+\frac{h^2}{\sigma^4})}+CtL^2\log(L)\big(\frac{h^2}{\sigma^2}+t\sigma^2+t^2h^2+t^3\sigma^4\big)\\
&\quad+CL^{4+\delta}\big(\frac{h^2}{\sigma^2}+t\sigma^2+t^2\sigma^2+t^3\sigma^4\big)
\end{align*}
for any $\delta>0$ and $C>0$ a constant depending only on $\eta$. In particular, we have 
\begin{equation*}
t\ll\frac{1}{h}\quad\implies\quad|r_k(t,\eta)|\ll tL^2\log(L)+L^4.
\end{equation*}
\end{proposition}

\begin{proof}
We have 
\begin{equation*}
\langle V^1\rangle_{K,\sigma} = (2\pi)^3 \sigma^2 \int_0^t \int_{\mathbb{R}^2} \overline{g_{K,\sigma}(x)}e^{-i s \Delta}\big( |e^{is\Delta} \varphi|^2 e^{is\Delta} \varphi \big)(x) \dd x \dd s.
\end{equation*}
and as $e^{i s \Delta}$ is selfadjoint, we have 
\begin{align*}
\langle V^1\rangle_{K,\sigma} &= (2\pi)^3 \sigma^2 \int_0^t \sum_{K_1,K_2,K_3} \eta_{K_1} \overline{\eta_{K_2}} \eta_{K_3}\int_{\mathbb{R}^2}  (\overline{e ^{is \Delta}g_{K,\sigma }})( e^{is\Delta} g_{K_1,h})( \overline{e^{is\Delta} g_{K_2,h}}) ( e^{is\Delta} g_{K_3,h}) \dd x\dd s \\
&= (2\pi)^3 \sigma^2 \sum_{K_1,K_2,K_3} \eta_{K_1}\overline{\eta_{K_2}}\eta_{K_3}\int_0^t W_{ K_1K_2 K_3 K}(s) \dd s
\end{align*}
where 
\begin{align*}
W_{ K_1K_2 K_3 K}(s) &=  \frac{1}{(2\pi)^8}|\beta_h(s)|^2 \beta_h(s) \overline{\beta_{\sigma }(s)}  \int_{\mathbb{R}^2} F_{K_1K_2K_3K}(s,x) \dd x 
\end{align*}
with 
\begin{align*}
 F_{K_1K_2K_3K}(s,x) = e^{- z(s)|x|^2 + \zeta(s) \cdot x + \gamma(s)} 
\end{align*}
and
\begin{align*}
z(s)&=h^2\beta_h(s)+\frac{h^2}{2}\overline{\beta_h(s)}+\frac{\sigma^2}{2}\overline{\beta_\sigma(s)},\\
\zeta(s)&=i\beta_h(s)(K_1 + K_3)-i\overline{\beta_h(s)}K_2-i\overline{\beta_\sigma(s)}K,\\
\gamma(s)&=-is\beta_h(s)(|K_1|^2+|K_3|^2)+is\overline{\beta_h(s)}|K_2|^2+is\overline{\beta_\sigma(s)}|K|^2.
\end{align*}
This gives
\begin{align*}
W_{ K_1K_2 K_3 K}(s) =\frac{\pi}{z(s)}\frac{1}{(2\pi)^8}|\beta_h(s)|^2 \beta_h(s) \overline{\beta_{\sigma }(s)}  e^{\gamma(s) + \frac{\zeta(s)\cdot \zeta(s)}{4 z(s)}}
\end{align*}
and we get 
\begin{align*}
\langle V^1\rangle_{K,\sigma}
 &= \frac{ \sigma^2}{(2\pi)^5} \sum_{K_1,K_2,K_3} \eta_{K_1}\overline{\eta_{K_2}}\eta_{K_3}
\int_0^t \frac{\pi}{z(s) } |\beta_h(s)|^2 \beta_h(s) \overline{\beta_{\sigma }(s)}  e^{ \gamma(s) + \frac{\zeta(s)\cdot \zeta(s)}{4 z(s)}} \dd s.
\end{align*}
We define  
the remainder 
\begin{align*}
&r_K(t,\eta)=(2\pi)^3\sigma^2\sum_{K\neq K_1-K_2+K_3}\eta_{K_1}\overline{\eta_{K_2}}\eta_{K_3}\int_0^tW_{K_1K_2K_3K}(s)\dd s\\
&\quad+(2\pi)^3\sigma^2\sum_{K=K_1-K_2+K_3}\eta_{K_1}\overline{\eta_{K_2}}\eta_{K_3}\int_0^t\Big(\frac{1}{(2\pi)^7\sigma^2}-W_{K_1K_2K_3K}(s)e^{is\Delta\omega_{KK_1K_2K_3}}\Big)e^{-is\Delta\omega_{KK_1K_2K_3}}\dd s\\
&=(2\pi)^3\sigma^2\sum_{K\neq K_1-K_2+K_3}\eta_{K_1}\overline{\eta_{K_2}}\eta_{K_3}\int_0^tW_{K_1K_2K_3K}(s)\dd s\\
&\quad+(2\pi)^3\sigma^2\sum_{\substack{K=K_1-K_2+K_3\\\Delta\omega_{K_1K_2K_3K}=0}}\eta_{K_1}\overline{\eta_{K_2}}\eta_{K_3}\int_0^t\Big(\frac{1}{(2\pi)^7\sigma^2}-W_{K_1K_2K_3K}(s)\Big)\dd s\\
&\quad+(2\pi)^3\sigma^2\sum_{\substack{K=K_1-K_2+K_3\\\Delta\omega_{K_1K_2K_3K}\neq0}}\eta_{K_1}\overline{\eta_{K_2}}\eta_{K_3}\frac{\Big(e^{it\Delta\omega_{KK_1K_2K_3}}W_{K_1K_2K_3K}(t)-\frac{1}{(2\pi)^7\sigma^2}\Big)e^{-it\Delta\omega_{KK_1K_2K_3}}}{i\Delta\omega_{KK_1K_2K_3}}\\
&\quad+(2\pi)^3\sigma^2\sum_{\substack{K=K_1-K_2+K_3\\\Delta\omega_{K_1K_2K_3K}\neq0}}\frac{\eta_{K_1}\overline{\eta_{K_2}}\eta_{K_3}}{i\Delta\omega_{KK_1K_2K_3}}\int_0^te^{-is\Delta\omega_{KK_1K_2K_3}}\partial_s(W_{K_1K_2K_3K}(s)e^{is\Delta\omega_{KK_1K_2K_3}})\dd s
\end{align*}
using an integration by part for the sum over nonresonant terms. We thus have
\begin{equation*}
\langle V^1(t)\rangle_{K,\sigma}=\frac{1}{(2\pi)^4}\sum_{K=K_1-K_2+K_3}\eta_{K_1}\overline{\eta_{K_2}}\eta_{K_3}\int_0^te^{-is\Delta\omega_{KK_1K_2K_3}}\dd s+R_K(t,\eta).
\end{equation*}
The result follows from the three following bounds. For $K\neq K_1-K_2+K_3$ and $s\le\frac{1}{\sigma^2L^2}$, we have
\begin{equation}\label{Estim1}\tag{A}
|W_{K_1K_2K_3K}(s)|\le Ce^{-\frac{1}{2\sigma^2L^2}+C(s+\frac{h^2}{\sigma^4})}
\end{equation}
for a positive constant $C>0$ depending only on $\eta$. For $K=K_1-K_2+K_3$ and $s\le\frac{1}{\sigma^2L^2}$, we have
\begin{equation}\label{Estim2}\tag{B}
\Big|\frac{1}{(2\pi)^7}-\sigma^2W_{K_1K_2K_3K}(s)e^{is\Delta\omega_{KK_1K_2K_3}}\Big|\le C\big(\frac{h^2}{\sigma^2}+s\sigma^2+s^2h^2+s^3\sigma^4\big)
\end{equation}
and
\begin{equation}\label{Estim3}\tag{C}
|\partial_s(W_{K_1K_2K_3K}(s)e^{is\Delta\omega_{KK_1K_2K_3}})|\le C(1+s).
\end{equation}
We now complete the proof before proving these bounds. The first bound gives
\begin{align*}
\Big|\sum_{K\neq K_1-K_2+K_3}\eta_{K_1}\overline{\eta_{K_2}}\eta_{K_3}\int_0^tW_{K_1K_2K_3K}(s)\dd s\Big|&\le CL^6\int_0^te^{-\frac{1}{2\sigma^2L^2}+C(s+\frac{h^2}{\sigma^4})}\dd s\\
&\le CL^6e^{-\frac{1}{2\sigma^2L^2}+C(t+\frac{h^2}{\sigma^4})}
\end{align*}
where $C>0$ denotes a constant depending only on $\eta$ that may change during the proof. The second bound gives
\begin{align*}
\Big|\sigma^2\sum_{\substack{K=K_1-K_2+K_3\\\Delta\omega_{K_1K_2K_3K}=0}}\eta_{K_1}\overline{\eta_{K_2}}\eta_{K_3}&\int_0^t\Big(\frac{1}{(2\pi)^7\sigma^2}-W_{K_1K_2K_3K}(s)\Big)\dd s\Big|\\
&\le C\sum_{\substack{K=K_1-K_2+K_3\\\Delta\omega_{K_1K_2K_3K}=0}}|\eta_{K_1}\overline{\eta_{K_2}}\eta_{K_3}|\int_0^t\big(\frac{h^2}{\sigma^2}+s\sigma^2+s^2h^2+s^3\sigma^4\big)\dd s\\
&\le CtL^2\log(L)\big(\frac{h^2}{\sigma^2}+t\sigma^2+t^2h^2+t^3\sigma^4\big)
\end{align*}
using that
\begin{equation*}
\frac{1}{L^2\log(L)}\sum_{\substack{K=K_1-K_2+K_3\\\Delta\omega_{K_1K_2K_3K}=0}}|\eta_{K_1}\overline{\eta_{K_2}}\eta_{K_3}|\le C
\end{equation*}
which follows from \cite{fgh} theorem $2.2$. The second bounds also gives
\begin{align*}
\Big|\sigma^2\sum_{\substack{K=K_1-K_2+K_3\\\Delta\omega_{K_1K_2K_3K}\neq0}}\eta_{K_1}\overline{\eta_{K_2}}\eta_{K_3}&\frac{\Big(e^{it\Delta\omega_{KK_1K_2K_3}}W_{K_1K_2K_3K}(t)-\frac{1}{(2\pi)^7\sigma^2}\Big)e^{-it\Delta\omega_{KK_1K_2K_3}}}{i\Delta\omega_{KK_1K_2K_3}}\Big|\\
&\le C\big(\frac{h^2}{\sigma^2}+t\sigma^2+t^2h^2+t^3\sigma^4\big)\sum_{\substack{K=K_1-K_2+K_3\\\Delta\omega_{K_1K_2K_3K}\neq0}}\Big|\frac{\eta_{K_1}\overline{\eta_{K_2}}\eta_{K_3}}{\Delta\omega_{KK_1K_2K_3}}\Big|.
\end{align*}
Proposition $3.9$ from \cite{fgh} gives
\begin{equation*}
\sum_{\substack{K=K_1-K_2+K_3\\\Delta\omega_{K_1K_2K_3K}=\xi}}|\eta_{K_1}\overline{\eta_{K_2}}\eta_{K_3}|\le CL^{2+\delta}
\end{equation*}
for any $\delta>0$, using that $\eta$ is compactly supported. We get
\begin{align*}
\sum_{\substack{K=K_1-K_2+K_3\\\Delta\omega_{K_1K_2K_3K}\neq0}}\Big|\frac{\eta_{K_1}\overline{\eta_{K_2}}\eta_{K_3}}{\Delta\omega_{KK_1K_2K_3}}\Big|&=\sum_{\xi\in\Z_{L^2}\cap[-A,A]}\frac{1}{|\xi|}\sum_{\substack{K=K_1-K_2+K_3\\\Delta\omega_{K_1K_2K_3K}=\xi}}|\eta_{K_1}\overline{\eta_{K_2}}\eta_{K_3}|\\
&\le CL^{2+\delta}\sum_{\xi\in\Z_{L^2}\cap[-A,A]}\frac{1}{|\xi|}\\
&\le CL^{4+\delta}\sum_{\xi\in\Z\cap[-AL^2,AL^2]}\frac{1}{|\xi|}\\
&\le CL^{4+\delta}\log(L)\\
&\le CL^{4+2\delta}
\end{align*}
for any $\delta>0$ and a finite $A>0$ depending on $\eta$ since it is compactly supported. The third bound gives
\begin{align*}
\Big|\sigma^2\sum_{\substack{K=K_1-K_2+K_3\\\Delta\omega_{K_1K_2K_3K}\neq0}}\frac{\eta_{K_1}\overline{\eta_{K_2}}\eta_{K_3}}{i\Delta\omega_{KK_1K_2K_3}}&\int_0^te^{-is\Delta\omega_{KK_1K_2K_3}}\partial_s(W_{K_1K_2K_3K}(s)e^{is\Delta\omega_{KK_1K_2K_3}})\dd s\Big|\\
&\le Ct(1+t)\sigma^2\sum_{\substack{K=K_1-K_2+K_3\\\Delta\omega_{K_1K_2K_3K}\neq0}}\Big|\frac{\eta_{K_1}\overline{\eta_{K_2}}\eta_{K_3}}{\Delta\omega_{KK_1K_2K_3}}\Big|\\
&\le Ct(1+t)\sigma^2L^{4+\delta}
\end{align*}
for any $\delta>0$. Thus we obtain
\begin{align*}
|r_K(t,\eta)|&\le CL^6e^{-\frac{1}{2\sigma^2L^2}+C(t+\frac{h^2}{\sigma^4})}+tL^2\log(L)\big(\frac{h^2}{\sigma^2}+t\sigma^2+t^2h^2+t^3\sigma^4\big)\\
&\quad+L^{4+\delta}\big(\frac{h^2}{\sigma^2}+t\sigma^2+t^2\sigma^2+t^3\sigma^4\big).
\end{align*}
For the last part of the statement, note that $tL^2\log(L)+L^4$ is of order $L^4$ for $t\le\frac{L^2}{\log(L)}$ and of order $tL^2\log(L)$ for $t\ge\frac{L^2}{\log(L)}$. In the first case, we have
\begin{equation*}
t\le\frac{L^2}{\log(L)}\quad\implies\quad th\ll t\sigma\le(hL^4)^{\frac{3}{4}}\frac{1}{L\log(L)}\ll 1
\end{equation*}
hence
\begin{equation*}
|r_K(t,\eta)|\le CL^6e^{-\frac{1}{3\sigma^2L^2}}+CL^4\left(\frac{1}{L^{2-\delta}}+\frac{\sigma(hL^4)^{\frac{3}{4}}}{L\log(L)}+\frac{(hL^4)^{\frac{3}{2}}}{L^{2-\delta}\log(L)^2}+\frac{\sigma(hL^4)^{\frac{9}{4}}}{L^3\log(L)^3}\right)
\end{equation*}
using $hL\ll\sigma\le h^{\frac{3}{4}}$ hence $|r_k(t,\eta)|\ll L^4+tL^2\log(L)$ and we now consider $t\ge\frac{L^2}{\log(L)}$. We have
\begin{equation*}
L^{4+\delta}t^2\sigma^2\ll tL^2\log(L)\quad\iff\quad t\ll\frac{1}{\sigma^2L^2}\frac{\log(L)}{L^\delta}.
\end{equation*}
Since $\sigma\le h^{\frac{3}{4}}$, we have
\begin{equation*}
\frac{1}{\sigma^2L^2}\frac{\log(L)}{L^\delta}\ge\frac{1}{h}\frac{1}{\sqrt{hL^4}}\frac{\log(L)}{L^\delta}\gg\frac{1}{h}
\end{equation*}
for $0<\delta<\delta_0$ since $hL^{4+\delta_0}\le1$ with $\delta_0>0$. Thus for $t\le\frac{\epsilon}{h}$, we have
\begin{equation*}
L^{4+\delta}t^2\sigma^2\ll tL^2\log(L)
\end{equation*}
hence
\begin{equation*}
|r_K(t,\eta)|\le CL^6e^{-\frac{1}{2\sigma^2L^2}+C(t+\frac{h^2}{\sigma^4})}+CtL^2\log(L)\big(\frac{1}{L^2}+\varepsilon\sqrt{h}+\varepsilon^2+\varepsilon^3\big).
\end{equation*}
It only remains to prove the bounds \eqref{Estim1}, \eqref{Estim2} and \eqref{Estim3}.

\medskip

\noindent\underline{\it Proof of \eqref{Estim1}.} Assume that $K\neq K_1-K_2+K_3$. For $s\le\frac{1}{\sigma^2L^2}$, we have $sh^2\ll s\sigma^2\le\frac{1}{L^2}\ll1$ hence
\begin{align*}
z(s)&=\frac{\sigma^2}{2}+\mathcal{O}(h^2+s\sigma^4),\\
\zeta(s)&=i(K_1+K_3-K_2-K)+\mathcal{O}(s\sigma^2)
\end{align*}
since $K_1,K_2,K_3,K\in B\cap\Z_L^2$ are bounded hence
\begin{align*}
\frac{1}{4z(s)}&=\frac{1}{2\sigma^2+\mathcal{O}(h^2+s\sigma^4)}\\
&=\frac{1}{2\sigma^2}\cdot\frac{1}{1+\mathcal{O}(h^2\sigma^{-2}+s\sigma^2)}\\
&=\frac{1}{2\sigma^2}+\mathcal{O}(h^2\sigma^{-4}+s)
\end{align*}
and
\begin{equation*}
\text{Re}\Big(\frac{\zeta(s)\cdot\zeta(s)}{4z(s)}\Big)=-\frac{|K_1+K_3-K_2-K|^2}{2\sigma^2}+\mathcal{O}(s+h^2\sigma^{-4}).
\end{equation*}
We also have $\gamma(s)=\mathcal{O}(s)$ hence
\begin{equation*}
\text{Re}\Big(\gamma(s)+\frac{\zeta(s)\cdot\zeta(s)}{4z(s)}\Big)=-\frac{|K_1+K_3-K_2-K|^2}{2\sigma^2}+\mathcal{O}(s+h^2\sigma^{-4})
\end{equation*}
which completes the proof.

\medskip

\noindent\underline{\it Proof of \eqref{Estim2}.} Consider
\begin{equation*}
\Gamma_{K_1K_2K_3K}(s):=(2\pi)^7\sigma^2W_{K_1K_2K_3K}(s)e^{is\Delta\omega_{KK_1K_2K_3}}
\end{equation*}
and we have to prove
\begin{equation*}
\Big|1-\Gamma_{K_1K_2K_3K}(s)\Big|\le C\big(\frac{h^2}{\sigma^2}+s\sigma^2+s^2h^2+s^3\sigma^4\big).
\end{equation*}
We have
\begin{equation*}
\Gamma_{K_1K_2K_3K}(s)=\frac{\sigma^2}{z(s)}|\beta_h(s)|^2 \beta_h(s) \overline{\beta_{\sigma }(s)}  e^{\widetilde\gamma(s) + \frac{\zeta(s)\cdot \zeta(s)}{4 z(s)}}
\end{equation*}
with
\begin{align*}
z(s)&=h^2\beta_h(s)+\frac{h^2}{2}\overline{\beta_h(s)}+\frac{\sigma^2}{2}\overline{\beta_\sigma(s)},\\
\zeta(s)&=i\beta_h(s)(K_1 + K_3)-i\overline{\beta_h(s)}K_2-i\overline{\beta_\sigma(s)}K,\\
\widetilde\gamma(s)&=-is(\beta_h(s)-1)(|K_1|^2+|K_3|^2)+is(\overline{\beta_h(s)}-1)|K_2|^2+is(\overline{\beta_\sigma(s)}-1)|K|^2.
\end{align*}
Again, for $s\le\frac{1}{\sigma^2L^2}$ we have $sh^2\ll s\sigma^2\ll1$ hence
\begin{equation*}
z(s)=\frac{\sigma^2}{2}+\frac{3h^2}{2}+\mathcal{O}(s\sigma^4)
\end{equation*}
and
\begin{align*}
\frac{1}{4z(s)}&=\frac{1}{2\sigma^2+6h^2+\mathcal{O}(s\sigma^4)}\\
&=\frac{1}{2\sigma^2}\big(1-\frac{3h^2}{\sigma^2}+\mathcal{O}(s\sigma^2)\big)\\
&=\frac{1}{2\sigma^2}-\frac{3h^2}{2\sigma^4}+\mathcal{O}(s).
\end{align*}
We also have
\begin{align*}
\zeta(s)&=i\beta_h(s)(K_1 + K_3)-i\overline{\beta_h(s)}K_2-i\overline{\beta_\sigma(s)}K\\
&=2sh^2(K_1+K_3+K_2)+2s\sigma^2 K+\mathcal{O}(s^2\sigma^4).
\end{align*}
We get
\begin{equation*}
\zeta(s)\cdot\zeta(s)=4s^2\sigma^4|K|^2+4s^2h^4|K_1+K_2+K_3|^2+8s^2h^2\sigma^2(K_1+K_2+K_3)\cdot K+\mathcal{O}(s^3\sigma^6)
\end{equation*}
hence
\begin{equation*}
\frac{\zeta(s)\cdot\zeta(s)}{4z(s)}=2s^2\sigma^2|K|^2+4s^2h^2(K_1+K_2+K_3)\cdot K-6s^2h^2|K|^2+\mathcal{O}(s^2\frac{h^4}{\sigma^2}+s^3\sigma^4).
\end{equation*}
We also have
\begin{equation*}
\widetilde\gamma(s)=-2h^2s^2(|K_1|^2+|K_2|^2+|K_3|^2)-2\sigma^2s^2|K|^2+\mathcal{O}(s^3\sigma^4)
\end{equation*}
hence
\begin{align*}
\widetilde\gamma(s)+\frac{\zeta(s)\cdot\zeta(s)}{4z(s)}&=-2h^2s^2(|K_1|^2+|K_2|^2+|K_3|^2)+4s^2h^2(K_1+K_2+K_3)\cdot K-6s^2h^2|K|^2\\
&\quad+\mathcal{O}(s^2\frac{h^4}{\sigma^2}+s^3\sigma^4)\\
&=-2h^2s^2|K_1-K|^2-2h^2s^2|K_2-K|^2-2h^2s^2|K_3-K|^2+\mathcal{O}(s^2\frac{h^4}{\sigma^2}+s^3\sigma^4)
\end{align*}
which gives
\begin{equation*}
e^{is(|K_1|^2+|K_3|^2-|K_2|^2-|K|^2)}e^{\widetilde\gamma(s)+\frac{\zeta(s)\cdot\zeta(s)}{4z(s)}}=1+\mathcal{O}(s^2h^2+s^3\sigma^4).
\end{equation*}
We get
\begin{align*}
\Gamma_{ K_1K_2 K_3 K}(s) &=\frac{\sigma^2}{z(s)}|\beta_h(s)|^2 \beta_h(s) \overline{\beta_{\sigma }(s)} \big(1+\mathcal{O}(s^2h^2+s^3\sigma^4)\big)\\
&=\frac{\sigma^2}{\sigma^2+\mathcal{O}(h^2+s\sigma^4)} \big(1+\mathcal{O}(sh^2)\big)\big(1+\mathcal{O}(s\sigma^2)\big)\big(1+\mathcal{O}(s^2h^2+s^3\sigma^4)\big)\\
&=\big(1+\mathcal{O}(s\sigma^2+h^2\sigma^{-2})\big)\big(1+\mathcal{O}(sh^2)\big)\big(1+\mathcal{O}(s\sigma^2)\big)\big(1+\mathcal{O}(s^2h^2+s^3\sigma^4)\big)\\
&=1+\mathcal{O}(\frac{h^2}{\sigma^2}+s\sigma^2+s^2h^2+s^3\sigma^4)
\end{align*}
which gives the bound on $1-\Gamma_{K_1K_2K_3K}(s)$. 

\medskip

\noindent\underline{\it Proof of \eqref{Estim3}.} For the derivatives, recall that
\begin{equation*}
\Gamma_{K_1K_2K_3K}(s)=\frac{\sigma^2}{z(s)}|\beta_h(s)|^2 \beta_h(s) \overline{\beta_{\sigma }(s)}  e^{\widetilde\gamma(s) + \frac{\zeta(s)\cdot \zeta(s)}{4 z(s)}}
\end{equation*}
with
\begin{align*}
z(s)&=h^2\beta_h(s)+\frac{h^2}{2}\overline{\beta_h(s)}+\frac{\sigma^2}{2}\overline{\beta_\sigma(s)},\\
\zeta(s)&=i\beta_h(s)(K_1 + K_3)-i\overline{\beta_h(s)}K_2-i\overline{\beta_\sigma(s)}K,\\
\widetilde\gamma(s)&=-is(\beta_h(s)-1)(|K_1|^2+|K_3|^2)+is(\overline{\beta_h(s)}-1)|K_2|^2+is(\overline{\beta_\sigma(s)}-1)|K|^2.
\end{align*}
Thus the derivative $\partial_s\Gamma_{K_1K_2K_3K}(s)$ is a sum of terms with the different derivatives
\begin{align*}
\partial_s\beta_\epsilon(s)&=\frac{4i\epsilon^2}{(1+2is\epsilon^2)^2},\\
\partial_s\Big(\frac{1}{z}\Big)(s)&=-\frac{z'(s)}{z(s)^2},\\
\partial_s\big(e^{\widetilde\gamma}\big)(s)&=\widetilde\gamma'(s)e^{\widetilde\gamma(s)},\\
\partial_s\big(e^{\frac{\zeta\cdot\zeta}{4z}}\big)(s)&=\frac{2\zeta'(s)\zeta(s)z(s)-\zeta(s)\zeta(s)z'(s)}{4z(s)^2}e^{\frac{\zeta(s)\cdot\zeta(s)}{4z(s)}}.
\end{align*}
Again, we work for $s\le\frac{1}{\sigma^2L^2}$ and we want to gain a small factor for each terms. We have
\begin{equation*}
\partial_s\beta_\epsilon(s)=\epsilon^2\big(1+\mathcal{O}(s\epsilon^2)\big)
\end{equation*}
hence we gain a factor $\epsilon^2$ with respect to $\beta_\epsilon(s)$. We have
\begin{align*}
\partial_s\Big(\frac{1}{z}\Big)(s)&=-\frac{z'(s)}{z(s)^2}\\
&=-\frac{h^2\beta_h'(s)+\frac{h^2}{2}\overline{\beta_h'(s)}+\frac{\sigma^2}{2}\overline{\beta_\sigma'(s)}}{\big(2\sigma^2+\mathcal{O}(h^2+s\sigma^4)\big)^2}\\
&\sim\frac{i\sigma^4}{4\sigma^4}
\end{align*}
hence this term is bounded. Since $\frac{1}{z}$ diverges as $\frac{1}{\sigma^2}$, we gain a factor $\sigma^2$. We have
\begin{align*}
\widetilde\gamma'(s)&=-i(\beta_h(s)-1)(|K_1|^2+|K_3|^2)+i(\overline{\beta_h(s)}-1)|K_2|^2+i(\overline{\beta_\sigma(s)}-1)|K|^2\\
&\quad-is\partial_s\beta_h(s)(|K_1|^2+|K_3|^2)+is\overline{\partial_s\beta_h(s)}|K_2|^2+is\overline{\partial_s\beta_\sigma(s)}|K|^2\\
&=2h^2s\beta_h(s)(|K_1|^2+|K_3|^2)+2s\overline{\beta_h(s)}|K_2|^2+2s\overline{\beta_\sigma(s)}|K|^2\\
&\quad-is\partial_s\beta_h(s)(|K_1|^2+|K_3|^2)+is\overline{\partial_s\beta_h(s)}|K_2|^2+is\overline{\partial_s\beta_\sigma(s)}|K|^2
\end{align*}
using that
\begin{equation*}
\beta_\epsilon(s)-1=\frac{2is\epsilon^2}{1+2is\epsilon^2}=2is\epsilon^2\beta_\epsilon(s)
\end{equation*}
hence we gain a factor $\sigma^2$ in the term
\begin{equation*}
\partial_s\big(e^{\widetilde\gamma}\big)(s)=\widetilde\gamma'(s)e^{\widetilde\gamma(s)}.
\end{equation*}
In particular, it is important that we removed the oscilating term $e^{-is(|K_1|^2+|K_3|^2-|K_2|^2-|K|^2)}$ here. We have
\begin{align*}
\frac{\zeta'(s)\zeta(s)}{2z(s)}=\big(i\partial_s\beta_h(s)(K_1+K_3)-i\overline{\partial_s\beta_h(s)}K_2-i\overline{\partial_s\beta_\sigma(s)}K\big)\frac{\zeta(s)}{2z(s)}
\end{align*}
hence we gain a factor $s\sigma^2$ using that
\begin{equation*}
\zeta(s)=i\beta_h(s)(K_1 + K_3)-i\overline{\beta_h(s)}K_2-i\overline{\beta_\sigma(s)}K=\mathcal{O}(s\sigma^2)
\end{equation*}
since $K=K_1-K_2+K_3$. Finally, we have the term
\begin{equation*}
\frac{\zeta(s)\zeta(s)z'(s)}{4z(s)^2}=\frac{z'(s)}{z(s)^2}\mathcal{O}(s^2\sigma^4)
\end{equation*}
where we gain a factor $s^2\sigma^4$ using the previous computations. Overall, this gives
\begin{equation*}
|\partial_s\Gamma_{K_1K_2K_3K}(s)|\lesssim(1+s)\sigma^2|\Gamma_{K_1K_2K_3K}(s)|
\end{equation*}
and the proof is complete.
\end{proof}

Now that we have an explicit bound on the error, we quantity
\begin{equation*}
\sum_{K=K_1-K_2+K_3}\eta_{K_1}\overline{\eta_{K_2}}\eta_{K_3}\frac{1-e^{-it\Delta\omega_{KK_1K_2K_3}}}{i\Delta\omega_{KK_1K_2K_3}}
\end{equation*}
in the asymptotic regime \eqref{regime1} for $t\ll\frac{1}{h}$. As done in the previous proof, proposition $3.9$ from \cite{fgh} implies
\begin{equation*}
\Big|\sum_{\substack{K=K_1-K_2+K_3\\\Delta\omega_{KK_1K_2K_3}\neq0}}\eta_{K_1}\overline{\eta_{K_2}}\eta_{K_3}\frac{1-e^{-it\Delta\omega_{KK_1K_2K_3}}}{i\Delta\omega_{KK_1K_2K_3}}\Big|\le CL^{4+\delta}
\end{equation*}
for $t\in\R$ and any $\delta>0$ with a constant $C=C(\eta,\delta)>0$. For the resonant sum, we have
\begin{equation*}
\sum_{\substack{K=K_1-K_2+K_3\\\Delta\omega_{KK_1K_2K_3}=0}}\eta_{K_1}\overline{\eta_{K_2}}\eta_{K_3}t=tL^2\log(L)\mathcal{R}_L(\eta)
\end{equation*}
with $\mathcal{R}_L(\eta)$ converging to $\mathcal{R}(\eta)$, see theorem $2.3$ in \cite{fgh}. Since
\begin{equation*}
L^{4+\delta}\ll tL^2\log(L)\quad\iff\quad t\gg\frac{L^{2+\delta}}{\log(L)},
\end{equation*}
we get
\begin{equation*}
\sum_{K=K_1-K_2+K_3}\eta_{K_1}\overline{\eta_{K_2}}\eta_{K_3}\frac{1-e^{-it\Delta\omega_{KK_1K_2K_3}}}{i\Delta\omega_{KK_1K_2K_3}}\simeq tL^2\log(L)\mathcal{R}(\eta)
\end{equation*}
for $L^{2+\delta}\le t\le\frac{1}{hL^\delta}$ for any $\delta>0$. For $t\le L^{2-\delta}$, the resonant sum is negligible and we have to study the limit of
\begin{equation*}
\sum_{\substack{K=K_1-K_2+K_3\\|K|^2\neq|K_1|^2-|K_2|^2+|K_3|^2}}\eta_{K_1}\overline{\eta_{K_2}}\eta_{K_3}\frac{1-e^{-it\Delta\omega_{KK_1K_2K_3}}}{i\Delta\omega_{KK_1K_2K_3}}.
\end{equation*}
The quasi-resonances dominate for time $t\ll L$ with a time-independent principal term while the resonances dominate for time $t\gg L^{2+\delta}$ with a linear growth in time. For times $t\le L$, this sums is a simple convergent Riemann sum while the behavior of the sum for $L\ll t\ll L^2$ in two dimensions is an open question. Recall that in order to use the expansion from theorem \ref{TheoremCaGa}, one needs to work with small initial data and that
\begin{equation*}
\CR_k(\eta)=\int_{\substack{k=k_1-k_2+k_3\\\Delta\omega_{kk_1k_2k_3}=0}}\eta(k_1)\overline{\eta(k_2)}\eta(k_3)\drm k_1\drm k_2\drm k_3
\end{equation*}
for $k\in\IR^2$.

\begin{theorem}
Let $(h,L,\sigma)$ be in the asymptotic scaling \eqref{regime1}, $K\in\Z_L^2$, $u(t)$ the solution to \eqref{NLS} with initial data $u(0)=\varepsilon\varphi$, $\delta>0$ and assume $\varepsilon\ll\frac{h^2}{L}$. Then $v(t)=e^{-it\Delta}u(t)$ satisfies
\begin{equation*}
\langle v(t)\rangle_{K,\sigma}=\frac{\sigma^2\varepsilon}{\sigma^2+h^2}\eta(K)-i\pi\frac{\varepsilon^3L^4}{(2\pi)^4}\CR_K(\eta)+\frac{\varepsilon^3L^4}{(2\pi)^4}\int_\R\frac{R_K(\xi)-\CR_K(\eta)}{\xi}\dd\xi+o(\varepsilon^3L^4)
\end{equation*}
for $L^\delta\le t\le L^{1-\delta}$. For $L^{2+\delta}\le t\ll\frac{1}{hL^\delta}$, we have 
\begin{equation*}
\langle v(t)\rangle_{K,\sigma}=\frac{\sigma^2\varepsilon}{\sigma^2+h^2}\eta(K)-i\pi\frac{2t\varepsilon^3L^2\log(L)}{\zeta(2)(2\pi)^4}\CR_K(\eta)+o\big(t\varepsilon^3L^2\log(L)\big).
\end{equation*}
\end{theorem}

\begin{proof}
For $L^{2+\delta}\le t\le\frac{1}{hL^\delta}$, we have
\begin{equation*}
\langle V^1(t,\varphi)\rangle_{K,\sigma}=-i\frac{tL^2\log(L)}{(2\pi)^4}\sum_{\substack{K=K_1-K_2+K_3\\\Delta\omega_{KK_1K_2K_3}=0}}\eta_{K_1}\overline\eta_{K_2}\eta_{K_3}+o\big(tL^2\log(L)\big)
\end{equation*}
using proposition \ref{PropExpansion1} and
\begin{equation*}
\Big|\sum_{\substack{K=K_1-K_2+K_3\\\Delta\omega_{KK_1K_2K_3}\neq0}}\eta_{K_1}\overline{\eta_{K_2}}\eta_{K_3}\frac{1-e^{-it\Delta\omega_{KK_1K_2K_3}}}{i\Delta\omega_{KK_1K_2K_3}}\Big|\le CL^{4+\frac{\delta}{2}}.
\end{equation*}
The result follows using the convergence from theorem $2.3$ in \cite{fgh}
\begin{equation*}
\lim_{L\to\infty}\frac{\zeta(2)}{2\log(L)L^2}\sum_{\substack{K=K_1-K_2+K_3\\\Delta\omega_{KK_1K_2K_3}=0}}\eta_{K_1}\overline\eta_{K_2}\eta_{K_3}=\CR_K(\eta,\overline\eta,\eta).
\end{equation*}
We have
\begin{align*}
\sum_{\substack{K=K_1-K_2+K_3\\\Delta\omega_{KK_1K_2K_3}\neq0}}\eta_{K_1}\overline{\eta_{K_2}}\eta_{K_3}\frac{1-e^{-it\Delta\omega_{KK_1K_2K_3}}}{i\Delta\omega_{KK_1K_2K_3}}&=\sum_{\substack{A,B\in\Z_L^2\\A\cdot B\neq 0}}\eta_{K+A}\overline{\eta_{K+A+B}}\eta_{K+B}\frac{1-e^{2itA\cdot B}}{iA\cdot B}\\
&=:2t\sum_{\substack{A,B\in\Z_L^2\\A\cdot B\neq 0}}\eta_{K+A}\overline{\eta_{K+A+B}}\eta_{K+B}f(2tA\cdot B)
\end{align*}
with
\begin{equation*}
f(x):=\frac{1-e^{ix}}{ix}.
\end{equation*}
We have
\begin{equation*}
f'(x)=\frac{xe^{-ix}+i(1-e^{-ix})}{x^2}
\end{equation*}
which is uniformly bounded hence
\begin{equation*}
\partial_{A_j} f(2tA\cdot B)=2tB_jf'(2tA\cdot B)\quad\text{and}\quad\partial_{B_j} f(2tA\cdot B)=2tA_jf'(2tA\cdot B)
\end{equation*}
for $j\in{1,2}$. Since the error of the Riemann approximation of a function $g$ is of order $L^{-1}\|\nabla g\|$, we get
\begin{align*}
\sum_{\substack{K=K_1-K_2+K_3\\\Delta\omega_{KK_1K_2K_3}\neq0}}&\eta_{K_1}\overline{\eta_{K_2}}\eta_{K_3}\frac{1-e^{-it\Delta\omega_{KK_1K_2K_3}}}{i\Delta\omega_{KK_1K_2K_3}}\\
&=L^4\int_{K=k_1-k_2+k_3}\eta(k_1)\overline{\eta(k_2)}\eta(k_3)\frac{1-e^{-it\Delta\omega_{Kk_1k_2k_3}}}{i\Delta\omega_{Kk_1k_2k_3}}\dd k_1\dd k_2\dd k_3+\mathcal{O}(L^3)
\end{align*}
for $t\le L^{1-\delta}$. With the co-area formula, we have
\begin{align*}
\int_{K=k_1-k_2+k_3}&\eta(k_1)\overline{\eta(k_2)}\eta(k_3)\frac{1-e^{-it\Delta\omega_{Kk_1k_2k_3}}}{i\Delta\omega_{Kk_1k_2k_3}}\dd k_1\dd k_2\dd k_3\\
&=\int_\R\frac{1-e^{it\xi}}{i\xi}\left(\int_{\substack{K=k_1-k_2+k_3\\\Delta\omega_{Kk_1k_2k_3}=\xi}}\eta(k_1)\overline{\eta(k_2)}\eta(k_3)\dd k_1\dd k_2\dd k_3\right)\dd \xi\\
&=\int_\R\frac{1-e^{it\xi}}{i\xi}R_K(\xi)\dd \xi\\
&=R_K(0)\int_\R\frac{\sin(\xi)}{\xi}\dd\xi+\int_\R\frac{R_K(\xi)-R_K(0)}{i\xi}\dd \xi-\int_\R\frac{R_K(\xi)-R_K(0)}{i\xi}e^{it\xi}\dd \xi
\end{align*}
with the compactly supported function
\begin{equation*}
R_K(\xi):=\int_{\substack{K=k_1-k_2+k_3\\\Delta\omega_{Kk_1k_2k_3}=\xi}}\eta(k_1)\overline{\eta(k_2)}\eta(k_3)\dd k_1\dd k_2\dd k_3.
\end{equation*}
Since $R_K$ is of class $\CC^\alpha$ for any $\alpha\in(0,1)$, $x\mapsto\frac{R_K(\xi)-R_K(0)}{i\xi}$ belongs to $L^1(\R)$ hence the proof is complete using Riemann-Lebesgue lemma.
\end{proof}




\section{Kinetic limit and random initial data}\label{SectionRandom}

We introduce the randomization of the initial data, compute $V_K^2(t)$ and the limit of the covariances. With the same setting as the previous section, we let $\zeta_K$ be a random variable of the form 
\begin{equation*}
\zeta_K:=\eta_K e^{i \theta_K}
\end{equation*}
with $(\theta_K)_{K\in\Z_L^2}$ are independent and identically distributed uniform random variables in $[0,2\pi]$. In the litterature, this is often refered to as {\em Random Phase} (RP). Our result would also hold for {\em Random Phase and Amplitudes} (RPA) with complex Gaussian random variables. We consider the randomization of the previous initial data
\begin{equation*}
\varphi(x):= \sum_{K\in\Z_L^2}\zeta_K g_{K,h}(x) = \frac{1}{(2\pi)^2} e^{- \frac{1}{2} h^2 |x|^2} \left( \sum_{K\in\Z_L^2}\zeta_K e^{i K \cdot x}\right)=: e^{- \frac{1}{2} h^2 |x|^2} F_L(x)
\end{equation*}
thus $\varphi$ is essentially a random periodic function $F_L$ with large period $2 \pi L$, with random phase Fourier coefficients, embedded in $\Sigma$ by Gaussian truncation. Since the randomness does not change the modulus of the Fourier transform of the function, it satisfies almost surely all the bounds from the previous sections. For $v(0)=\eps\varphi$, we compute the second order expansion of the variance
\begin{equation*}
\E\big[|\langle v(t)\rangle_{K,\sigma}|^2\big]=\eps^2E_{K,\sigma}^0(t,\varphi)+\eps^4E_{K,\sigma}^1(t,\eta)+\eps^6E_{K,\sigma}^2(t,\eta)+R_{K,\sigma}(t,\eta)
\end{equation*}
with
\begin{align*}
E_{K,\sigma}^0(t,\eta)&=\E\big[|\langle \varphi\rangle_{K,\sigma}|^2\big],\\
E_{K,\sigma}^1(t,\eta)&=2\E\big[\Re\big(\langle v\rangle_{K,\sigma}\overline{\langle V^1\rangle}_{K,\sigma}\big)\big],\\
E_{K,\sigma}^2(t,\eta)&=\E\big[|\langle V^1\rangle_{K,\sigma}|^2\big]+2\E\big[\Re\big(\langle v\rangle_{K,\sigma}\overline{\langle V^2\rangle}_{K,\sigma}\big)\big].
\end{align*}
With the previous computation of $\langle V^1\rangle_{K,\sigma}$ in the deterministic case, the main computation to perform is $\langle V^2\rangle_{K,\sigma}$. Again, we work in the timeframe $t\epsilon^2\ll1$ thus $\beta_\epsilon(t)\simeq 1$ and
\begin{equation*}
(e^{i t \Delta} g_{K,\epsilon})(x) \simeq \frac{1}{(2\pi)^2}  e^{  - \frac{\epsilon^2}{2}|x|^2   
+ i  x \cdot K
-  i t |K|^2}=g_{K,\epsilon}(x)e^{-it|K|^2}.
\end{equation*}
We have
\begin{equation*}
V_k^2(t)=2\int_0^t \int_{0}^{s} R_k(s,\varphi,\varphi,R(s',\varphi,\varphi,\varphi)) \dd s'\dd s +\int_0^t\int_{0}^{s} R_k(s, \varphi, R(s',\varphi,\varphi,\varphi),\varphi) \dd s'\dd s
\end{equation*}
and
\begin{align*}
\int_0^t \int_{0}^{s} &R_k(s,g_{K_1,h},g_{K_2,h},R(s',g_{K_3,h},g_{K_4,h},g_{K_5,h})) \dd s'\dd s\\
&=\int_0^t\int_0^s e^{-is\Delta}\Big(e^{is\Delta}g_{K_1,h}e^{-is\Delta}\overline{g_{K_2,h}}e^{i(s-s')\Delta}\big(e^{is'\Delta}g_{K_3,h}e^{-is'\Delta}\overline{g_{K_4,h}}e^{is'\Delta}g_{K_5,h}\big)\Big)\dd s'\dd s\\
&\simeq\int_0^t\int_0^s e^{-is(|K_1-K_2+K_6|^2-|K_1|^2+|K_2|^2-|K_6|^2)}e^{-is'(|K_6|^2-|K_3|^2+|K_4|^2-|K_5|^2)}\dd s'\dd s
\end{align*}
with $K_6=K_3-K_4+K_5$ as for the first order term and a localization on $K=K_1-K_2+K_6$ for $\langle V^2\rangle_{K,\sigma}$. This gives the following proposition.

\begin{proposition}\label{PropExpansion2}
Let $(h,L,\sigma)$ in asymptotic regime \eqref{regime1}. Then we have 
\begin{align*}
\langle V^2\rangle_{K,\sigma} 
&=2\sum_{\substack{K=K_1-K_2+K_3\\ K_1=K_4-K_5+K_6}}\eta_{K_4}\overline{\eta_{K_5}}\eta_{K_6}\overline{\eta_{K_2}}\eta_{K_3}\int_0^t\int_0^se^{is\Delta\omega_{K_1K_2K_3K}}e^{-is'\Delta\omega_{K_4K_5K_6K_1}}\dd s'\dd s\\
&\quad+\sum_{\substack{K=K_1-K_2+K_3\\ K_2=K_4-K_5+K_6}}\eta_{K_1}\overline{\eta_{K_4}}\eta_{K_5}\overline{\eta_{K_6}}\eta_{K_3}\int_0^t\int_0^se^{is\Delta\omega_{K_1K_2K_3K}}e^{-is'\Delta\omega_{K_4K_5K_6K_2}}\dd s'\dd s\\
&\quad+r_K(t,\eta)
\end{align*}
with
\begin{equation*}
t\ll\frac{1}{h}\quad\implies\quad r_K(t,\eta)\ll tL^2\log(L)+L^4.
\end{equation*}
\end{proposition}

\begin{proof}
We have
\begin{align*}
V_k^2(t)&=2\int_0^t \int_{0}^{s} R_k(s,\varphi,\varphi,R(s',\varphi,\varphi,\varphi)) \dd s'\dd s +\int_0^t\int_{0}^{s} R_k(s, \varphi, R(s',\varphi,\varphi,\varphi),\varphi) \dd s'\dd s\\
&=:\int_0^t\int_0^s(2A+B)(s,s')\dd s'\dd s. 
\end{align*}
For the first term, we have
\begin{align*}
&\langle A\rangle_{K,\sigma} \\
&=2\pi\sigma^2\sum_{K_1,K_2,K_3,K_4,K_5}\int_0^t\int_0^s\int_{\R^2}\overline{g_{K,\sigma}}e^{-is\Delta}\Big(e^{is\Delta}g_{K_1,h}\overline{e^{is\Delta}g_{K_2,h}}e^{i(s-s')\Delta}\big(e^{is'\Delta}g_{K_3,h}\overline{e^{is'\Delta}g_{K_4,h}}e^{is'\Delta}g_{K_5,h}\big)\Big)\\
&=2\pi\sigma^2\sum_{K_1,K_2,K_3,K_4,K_5}\int_0^t\int_0^s\int_{\R^2}\overline{e^{is\Delta}g_{K,\sigma}}e^{is\Delta}g_{K_1,h}\overline{e^{is\Delta}g_{K_2,h}}e^{i(s-s')\Delta}\big(e^{is'\Delta}g_{K_3,h}\overline{e^{is'\Delta}g_{K_4,h}}e^{is'\Delta}g_{K_5,h}\big)\\
&=2\pi\sigma^2\sum_{K_1,K_2,K_3,K_4,K_5}\int_0^t\int_0^s\int_{\R^2}e^{is\Delta}\big(\overline{e^{is\Delta}g_{K,\sigma}}e^{is\Delta}g_{K_1,h}\overline{e^{is\Delta}g_{K_2,h}}\big)e^{-is'\Delta}\big(e^{is'\Delta}g_{K_3,h}\overline{e^{is'\Delta}g_{K_4,h}}e^{is'\Delta}g_{K_5,h}\big).
\end{align*}
Using that
\[
(e^{i t \Delta} g_{K,\epsilon})(x) = \frac{\beta_\epsilon(t)}{(2\pi)^2}  e^{  - \frac{\epsilon^2\beta_\epsilon(t)  }{2}|x|^2   
+ i \beta_{\epsilon}(t) x \cdot K
-  i t \beta_{\epsilon}(t) |K|^2}
\]
with 
\[
\beta_{\epsilon}(t) = \frac{1}{1 + 2 it \epsilon^2},
\]
we get
\begin{equation*}
\overline{e^{is\Delta}g_{K,\sigma}}e^{is\Delta}g_{K_1,h}\overline{e^{is\Delta}g_{K_2,h}}=\frac{1}{(2\pi)^6}|\beta_h(s)|^2\overline{\beta_\sigma(s)}e^{- z(s)|x|^2 + \zeta(s) \cdot x + \gamma(s)}
\end{equation*}
and
\begin{equation*}
e^{is'\Delta}g_{K_3,h}\overline{e^{is'\Delta}g_{K_4,h}}e^{is'\Delta}g_{K_5,h}=\frac{1}{(2\pi)^6}|\beta_h(s')|^2\beta_h(s')e^{-\widetilde z(s')|x|^2 + \widetilde\zeta(s') \cdot x + \widetilde\gamma(s')}
\end{equation*}
with
\begin{align*}
z(s)&=h^2\text{Re}\big(\beta_h(s)\big)+\frac{\sigma^2}{2}\overline{\beta_\sigma(s)},\\
\zeta(s)&=-i\overline{\beta_\sigma(s)}K+i\beta_h(s)K_1-i\overline{\beta_h(s)}K_2,\\
\gamma(s)&=is\overline{\beta_\sigma(s)}|K|^2-is\beta_h(s)|K_1|^2+is\overline{\beta_h(s)}|K_2|^2,\\
\widetilde z(s')&=h^2\text{Re}\big(\beta_h(s')\big)+\frac{h^2}{2}\overline{\beta_h(s')},\\
\widetilde\zeta(s')&=i\beta_h(s')K_3-i\overline{\beta_h(s')}K_4+i\beta_h(s')K_5,\\
\widetilde\gamma(s')&=-is\beta_h(s')|K_3|^2+is\overline{\beta_h(s')}|K_4|^2-is\beta_h(s')|K_5|^2.
\end{align*}
Recall that for
\[
f(x) = e^{- z |x|^2 + \xi \cdot x},
\]
we have 
\begin{equation*}
(e^{i t \Delta} f )(x) = \frac{1}{1 + 4izt}  e^{  - \frac{z}{1 + 4i tz}|x|^2 
+ \frac{1}{1 + 4i tz} x \cdot \xi 
+  \frac{it }{1 + 4i tz}  \xi \cdot \xi}
\end{equation*}
thus
\begin{align*}
e^{is\Delta}(\overline{e^{is\Delta}g_{K,\sigma}}e^{is\Delta}g_{K_1,h}\overline{e^{is\Delta}g_{K_2,h}})&=\frac{|\beta_h(s)|^2\overline{\beta_\sigma(s)}}{(2\pi)^6\big(1+4isz(s)\big)}\\
&\quad\times e^{-\frac{z(s)}{1 + 4isz(s)}|x|^2 
+ \frac{1}{1 + 4i sz(s)} x \cdot \zeta(s) 
+  \frac{is }{1 + 4i sz(s)}  \zeta(s) \cdot \zeta(s)}e^{\gamma(s)}
\end{align*}
and
\begin{align*}
e^{-is'\Delta}(e^{is'\Delta}g_{K_3,h}\overline{e^{is'\Delta}g_{K_4,h}}e^{is'\Delta}g_{K_5,h})&=\frac{|\beta_h(s')|^2\beta_h(s')}{(2\pi)^6\big(1-4is'\widetilde z(s')\big)}\\
&\quad\times e^{-\frac{\widetilde z(s')}{1 - 4is'\widetilde z(s')}|x|^2 
+ \frac{1}{1 - 4i s'\widetilde z(s')} x \cdot \widetilde\zeta(s') 
-  \frac{is' }{1 - 4i s'\widetilde z(s')}  \widetilde\zeta(s') \cdot \widetilde\zeta(s')}e^{\widetilde\gamma(s')}.
\end{align*}
This gives
\begin{equation*}
\langle A\rangle_{K,\sigma}=2\pi\sigma^2\sum_{K_1,K_2,K_3,K_4,K_5}\int_0^t\int_0^s W_{K_1K_2K_3K_4K_5K}(s,s')\dd s\dd s'
\end{equation*}
with
\begin{align*}
W_{K_1K_2K_3K_4K_5K}(s,s')&:=\frac{|\beta_h(s)|^2\overline{\beta_\sigma(s)}|\beta_h(s')|^2\beta_h(s')}{(2\pi)^{12}\big(1+4isz(s)\big)\big(1-4is'\widetilde z(s')\big)}e^{c(s,s')+\frac{b(s,s')\cdot b(s,s')}{4a(s,s')}}
\end{align*}
and
\begin{align*}
a(s,s')&:=\frac{z(s)}{1 + 4isz(s)}+\frac{\widetilde z(s')}{1 - 4is'\widetilde z(s')},\\
b(s,s')&:=\frac{\zeta(s)}{1 + 4i sz(s)}+\frac{\widetilde\zeta(s')}{1 - 4i s'\widetilde z(s')},\\
c(s,s')&:=\frac{is }{1 + 4i sz(s)}  \zeta(s) \cdot \zeta(s)+\frac{is' }{1 - 4i s'\widetilde z(s')}  \widetilde\zeta(s') \cdot \widetilde\zeta(s')+\gamma(s)+\widetilde\gamma(s').
\end{align*}
Recall that we work in the scaling $(s+s')h^2\ll(s+s')\sigma^2\ll1$. We have
\begin{align*}
z(s)&=h^2\text{Re}\big(\beta_h(s)\big)+\frac{\sigma^2}{2}\overline{\beta_\sigma(s)}\\
&=h^2+\frac{\sigma^2}{2}+\mathcal{O}(s\sigma^4),\\
\zeta(s)&=-i\overline{\beta_\sigma(s)}K+i\beta_h(s)K_1-i\overline{\beta_h(s)}K_2\\
&=-i(K-K_1+K_2)+2s\sigma^2K+2sh^2(K_1+K_2)+\mathcal{O}(s^2\sigma^4),\\
\gamma(s)&=is\overline{\beta_\sigma(s)}|K|^2-is\beta_h(s)|K_1|^2+is\overline{\beta_h(s)}|K_2|^2\\
&=is(|K|^2-|K_1|^2+|K_2|^2)-2s^2\sigma^2|K|^2-2s^2h^2(|K_1|^2+|K_2|^2)+\mathcal{O}(s^3\sigma^4),\\
\widetilde z(s')&=h^2\text{Re}\big(\beta_h(s')\big)+\frac{h^2}{2}\overline{\beta_h(s')}\\
&=\frac{3h^2}{2}+\mathcal{O}(sh^4),\\
\widetilde\zeta(s')&=i\beta_h(s')K_3-i\overline{\beta_h(s')}K_4+i\beta_h(s')K_5\\
&=i(K_3-K_4+K_5)+sh^2(K_3+K_4+K_5)+\mathcal{O}(s^2h^4),\\
\widetilde\gamma(s')&=-is\beta_h(s')|K_3|^2+is\overline{\beta_h(s')}|K_4|^2-is\beta_h(s')|K_5|^2\\
&=-is(|K_3|^2-|K_4|^2+|K_5|^2)-s^2h^2(|K_3|^2+|K_4|^2+|K_5|^2)+\mathcal{O}(s^3h^4).
\end{align*}
We have
\begin{equation*}
sz(s)=\mathcal{O}(s\sigma^2)\quad\text{and}\quad s'\widetilde z(s')=\mathcal{O}(s'h^2)
\end{equation*}
hence we have
\begin{align*}
a(s,s')&=z(s)\big(1+\mathcal{O}(s\sigma^2)\big)+\widetilde z(s')\big(1+\mathcal{O}(s'h^2)\big)\\
&=\frac{1}{2}\sigma^2+\frac{5}{2}h^2+\mathcal{O}(s\sigma^2+s'h^2)\\
b(s,s')&=\zeta(s)\big(1+\mathcal{O}(s\sigma^2)\big)+\widetilde\zeta(s')\big(1+\mathcal{O}(s\sigma^2)\big)\\
&=-i(K-K_1+K_2-K_3+K_4-K_5)+\mathcal{O}(s\sigma^2),\\
c(s,s')&=\mathcal{O}(s).
\end{align*}
As for the first order term, this gives an exponential localisation on $K=K_1-K_2+K_3-K_4+K_5$. We get
\begin{align*}
\int_0^t\int_0^s&e^{is(|K|^2-|K_1|^2+|K_2|^2)}e^{-is'(|K_3|^2-|K_4|^2+|K_5|^2)}e^{-is|K-K_1+K_2|^2}e^{is'|K_3-K_4+K_5|^2}\dd s'\dd s\\
&=\int_0^t\int_0^se^{is(|K|^2-|K_1|^2+|K_2|^2-|K_6|^2)}e^{-is'(|K_3|^2-|K_4|^2+|K_5|^2-|K_6|^2)}\dd s'\dd s
\end{align*}
with $K_6:=K-K_1+K_2=K_3-K_4+K_5$. The rest of the proof follows from similar computations and a bound on the derivative with an integration by part, as for the proof of Proposition \ref{PropExpansion1}.
\end{proof}

We get the expansion
\begin{align*}
&\langle v(t)\rangle_{K,\sigma}=\varepsilon\zeta(K)-\frac{\eps^3}{(2\pi)^4}\sum_{K=K_1-K_2+K_3}\zeta_{K_1}\overline{\zeta_{K_2}}\zeta_{K_3}\frac{1-e^{-it\Delta\omega_{KK_1K_2K_3}}}{{\Delta\omega_{KK_1K_2K_3}}}\\
&\quad-\frac{\eps^5}{(2\pi)^8}2\sum_{\substack{K=K_1-K_2+K_3\\K_1=K_4-K_5+K_6}}\frac{\zeta_{K_4}\overline{\zeta_{K_5}}\zeta_{K_6}\overline{\zeta_{K_2}}\zeta_{K_3}}{\Delta\omega_{K_1K_4K_5K_6}}\left(\frac{e^{-it(\Delta\omega_{KK_1K_2K_3}-\Delta\omega_{K_1K_4K_5K_6})}-1}{\Delta\omega_{KK_1K_2K_3}-\Delta\omega_{K_1K_4K_5K_6}}-\frac{e^{-it\Delta\omega_{K_1K_4K_5K_6}}-1}{\Delta\omega_{KK_1K_2K_3}}\right)\\
&\quad-\frac{\eps^5}{(2\pi)^8}\sum_{\substack{K=K_1-K_2+K_3\\K_2=K_4-K_5+K_6}}\frac{\zeta_{K_1}\overline{\zeta_{K_4}}\zeta_{K_5}\overline{\zeta_{K_6}}\zeta_{K_3}}{\Delta\omega_{K_2K_4K_5K_6}}\left(\frac{e^{-it(\Delta\omega_{KK_1K_2K_3}-\Delta\omega_{K_2K_4K_5K_6})}-1}{\Delta\omega_{KK_1K_2K_3}-\Delta\omega_{K_2K_4K_5K_6}}-\frac{e^{-it\Delta\omega_{K_2K_4K_5K_6}}-1}{\Delta\omega_{KK_1K_2K_3}}\right)\\
&\quad+r_{K,\sigma}(t)
\end{align*}
using that for $a,b\in\R$, we have
\begin{equation*}
\int_0^t\int_0^se^{isa}e^{-is'b}\dd s'\dd s=\frac{1}{b}\left(\frac{e^{it(a-b)}-1}{(a-b)}-\frac{e^{ita}-1}{a}\right)
\end{equation*}
with the conventions $\frac{e^{ita}-1}{a}=it$ for $a=0$ and $\frac{1}{b}\big(\frac{e^{it(a-b)}-1}{a-b}-\frac{e^{ita}-1}{a}\big)=-\frac{1}{2}t^2$ for $a=b=0$. This is exactly the formal computations done by Buckmaster, Germain, Hani and Shatah in \cite{BGHS} Section $3$ and they prove that taking the expectation of yields the kinetic operator $\CT_K$. While the main term in $E_1(t,\eta)$ vanishes at first order, it still dominates the second order term in the expansion with respec to $\varepsilon$. Because the nonlinearity satisfies $g(\overline u)=\overline{g(u)}$, $E_1(t,\eta)=-E_1(-t,\eta)$ and we have
\begin{equation*}
\IE\big[|\langle v(t)\rangle_{K,\sigma}|^2\big]+\IE\big[|\langle v(-t)\rangle_{K,\sigma}|^2\big]=2\varepsilon^2|\eta(K)|^2+2\eps^6E_2(t,\eta)+\mathcal{O}(\varepsilon^8).
\end{equation*}

\begin{theorem}
Let $(h,L,\sigma)$ be in the asymptotic scaling \eqref{regime1}, $K\in\Z_L^2$, $u(t)$ the solution to \eqref{NLS} with random initial data $u(0)=\varepsilon\varphi_\theta$, $\delta>0$ and assume $\varepsilon\ll\frac{h^2}{L}$. Then $v(t)=e^{-it\Delta}u(t)$ satisfies
\begin{equation*}
\IE\big[|\langle v(t)\rangle_{K,\sigma}|^2\big]=\frac{\sigma^4\varepsilon^2}{(\sigma^2+h^2)^2}|\eta(K)|^2+\varepsilon^4E_1(t,\eta)+\frac{t\varepsilon^6L^4}{(2\pi)^8}\CT_K(\eta)+o(t\varepsilon^5L^4)
\end{equation*}
for $L^\delta\le t\le L^{1-\delta}$. For $L^{2+\delta}\le t\le\frac{1}{hL^\delta}$, we have 
\begin{equation*}
\IE\big[|\langle v(t)\rangle_{K,\sigma}|^2\big]=\frac{\sigma^4\varepsilon^2}{(\sigma^2+h^2)^2}|\eta(K)|^2+\varepsilon^4E_1(t,\eta)+\frac{2t^2\varepsilon^6L^2\log(L)}{\zeta(2)(2\pi)^8}\CT_K(\eta)+o\big(t^2\varepsilon^5L^2\log(L)\big).
\end{equation*}
Moreover, we have
\begin{equation*}
\IE\Big[\langle v(t)\rangle_{K,\sigma}\overline{\langle v(t)\rangle_{K',\sigma}}\Big]=o\big(\varepsilon^4L^4+t\varepsilon^4L^2\log(L)\big)
\end{equation*}
for $K\neq K'$.
\end{theorem}

\begin{proof}
Recall that
\begin{equation*}
\E\big[|\langle v(t)\rangle_{K,\sigma}|^2\big]=\eps^2E_{K,\sigma}^0(t,\varphi)+\eps^4E_{K,\sigma}^1(t,\varphi)+\eps^6E_{K,\sigma}^2(t,\varphi)+R_{K,\sigma}(t,\varphi)
\end{equation*}
with
\begin{align*}
E_{K,\sigma}^0(t,\varphi)&=\E\big[|\langle \varphi\rangle_{K,\sigma}|^2\big],\\
E_{K,\sigma}^1(t,\varphi)&=2\E\big[\Im\big(\langle v\rangle_{K,\sigma}\overline{\langle V^1\rangle}_{K,\sigma}\big)\big],\\
E_{K,\sigma}^2(t,\varphi)&=\E\big[|\langle V^1\rangle_{K,\sigma}|^2\big]+2\E\big[\Im\big(\langle v\rangle_{K,\sigma}\overline{\langle V^2\rangle}_{K,\sigma}\big)\big].
\end{align*}
In the second order term, the expectation
\begin{equation*}
\IE\big[\zeta_{K_1}\zeta_{K_2}\zeta_{K_3}\overline{\zeta_{K_4}\zeta_{K_5}\zeta_{K_6}}\big]\in\{0,1\}
\end{equation*}
gives $6$ terms with $(K_1,K_2,K_3)$ a permutation of $(K_4,K_5,K_6)$. For the square term, similar computations as Proposition \ref{PropExpansion1} gives
\begin{align*}
&\E\big[|\langle V^1\rangle_{K,\sigma}|^2\big]\\
&=4t^2|\eta_{K}|^2\sum_{K_1,K_2}|\eta_{K_1}|^2|\eta_{K_2}|^2+2\sum_{K=K_1-K_2+K_3}|\eta_{K_1}|^2|\eta_{K_2}|^2|\eta_{K_3}|^2\Big|\frac{1-e^{-it\Delta\omega_{KK_1K_2K_3}}}{\Delta\omega_{KK_1K_2K_3}}\Big|^2\\
&\quad+o(tL^2\log(L)+L^4)\\
&=4t^2|\eta_{K}|^2\sum_{K_1,K_2}|\eta_{K_1}|^2|\eta_{K_2}|^2+2\sum_{K=K_1-K_2+K_3}|\eta_{K_1}|^2|\eta_{K_2}|^2|\eta_{K_3}|^2\Big|\frac{\sin(\frac{1}{2}t\Delta\omega_{KK_1K_2K_3})}{\frac{1}{2}\Delta\omega_{KK_1K_2K_3}}\Big|^2\\
&\quad+o(t^2L^2\log(L)+tL^4)
\end{align*}
for $t=o(h^{-1})$. Following the computation in \cite{BGHS} Section $3$ and Proposition \ref{PropExpansion2}, we get
\begin{align*}
&E_2(t,\eta)\\&=\sum_{K=K_1-K_2+K_3}|\eta_K|^2|\eta_{K_1}|^2|\eta_{K_2}|^2|\eta_{K_3}|^2\Big(\frac{1}{|\eta_{K}|^2}-\frac{1}{|\eta_{K_1}|^2}+\frac{1}{|\eta_{K_2}|^2}-\frac{1}{|\eta_{K_3}|^2}\Big)\Big|\frac{\sin(\frac{1}{2}t\Delta\omega_{KK_1K_2K_3})}{\frac{1}{2}\Delta\omega_{KK_1K_2K_3}}\Big|^2\\
&\quad+o(t^2L^2\log(L)+tL^4)
\end{align*}
which is again a convergence Riemann sum for $t\le L^{1-\delta}$. For $L^{2+\delta}\le t\le\frac{1}{hL^\delta}$, the result follows with similar proof as done in \cite{fgh}. For the propagation of chaos, the term of order $\eps^2$ cancels using that $\IE\big[\zeta_K\overline\zeta_{K'}\big]=0$ for $K\neq K'$. To complete the proof, one only needs to prove that the term $E_1(t,\eta)$ cancels at first order which again follows from the computation in \cite{BGHS}.
\end{proof}


\appendix

\section{Appendix}

We compute here the Fourier transform of complex Gaussians.

\begin{lemma}
\label{fouriertransformgaussians}
Let 
\begin{equation}
\label{conditionz}
z\in \Bbb C, \ \ \Re(z)\ge 0. 
\end{equation}
Then (in dimension 2)
\begin{equation}
\label{foueriergaussinees}
\mathcal F\left(e^{-z|\cdot|^2}\right)=\frac{\pi}{z}e^{-\frac{|\xi|^2}{4z}}.
\end{equation}
More generally, let $z=z_1+iz_2\in \Bbb C$, $a\in \Bbb R^2$ and $$f(x)=e^{-z|x|^2+a\cdot x},$$ then 
\begin{equation}
\label{fouriergausian}
\hat{f}(k)=\frac{\pi}{z}e^{-\frac{|k|^2}{4z}-\frac{i(a\cdot k)}{2z}+\frac{|a|^2}{4z}}. 
\end{equation}
and the equivalent formula: for any $\xi = (\xi_1,\xi_2) \in \C^2$, with the notation $\xi \cdot x = \xi_1 x_1 + \xi_2 x_2$ and $\xi\cdot \xi = \xi_1^2 + \xi_2^2 \in \C$, 
\begin{equation}
\label{magic}
\int_{\R^d} e^{-z|x|^2+ \xi \cdot x} \dd x = \frac{\pi}{z} e^{ \frac{\xi \cdot \xi}{4z}}. 
\end{equation}
\end{lemma}

\begin{proof}
Indeed,
$$
-z|x|^2+a\cdot x= -z\left|x-\frac{a}{2z_1}\right|^2-i\frac{z_2 a}{z_1}\cdot x+\frac{z|a|^2}{4z_1^2}
$$
We now recall
$$\widehat{f(x-a)}(k)=\hat{f}(k)e^{-ia\cdot k }, \ \ \widehat{fe^{-ia\cdot x}}=\hat{f}(k+a).$$ Let $$g(x)=e^{-z\left|x-\frac{a}{2z_1}\right|^2},$$ then $$\hat{g}(k)=\left(\frac{\pi}{z}\right)^{\frac d2}e^{-\frac{|k|^2}{4z}-i\frac{a}{2z_1}\cdot k}$$
and $$f(x)=g(x)e^{-i\frac{z_2 a}{z_1}\cdot x+\frac{z|a|^2}{4z_1^2}}$$
and hence
\begin{eqnarray*}
\hat{f}(k)&=&e^{\frac{z|a|^2}{4z_1^2}}\hat{g}\left(k+\frac{z_2}{z_1}a\right)=\left(\frac{\pi}{z}\right)^{\frac d2}e^{\frac{z|a|^2}{4z_1^2}}e^{-\frac{1}{4z}\left|k+\frac{z_2}{z_1}a\right|^2-i\frac{a}{2z_1}\cdot\left(k+\frac{z_2}{z_1}a\right)}\\
& = & \left(\frac{\pi}{z}\right)^{\frac d2}e^{\frac{z|a|^2}{4z_1^2}-\frac{1}{4z}\left(|k|^2+\frac{2z_2}{z_1}a\cdot k+\frac{z_2^2}{z_1^2}|a|^2\right)-\frac{ia\cdot k}{2z_1}-\frac{i|a|^2z_2}{2z_1^2}}\\
& = &  \left(\frac{\pi}{z}\right)^{\frac d2}e^{-\frac{|k|^2}{4z}+(a\cdot k)\left(-\frac{z_2}{2zz_1}-\frac{i}{2z_1}\right)+|a|^2\left(\frac{z}{4z_1^2}-\frac{z_2^2}{4zz_1^2}-\frac{iz_2}{2z_1^2}\right)}\\
& = &  \left(\frac{\pi}{z}\right)^{\frac d2}e^{-\frac{|k|^2}{4z}-\frac{a\cdot k}{2zz_1}\left[z_2+i(z_1+iz_2)\right]+\frac{|a|^2}{4zz_1^2}\left[(z_1+iz_2)^2-z_2^2-2iz_2(z_1+iz_2)\right]}\\
& =& \left(\frac{\pi}{z}\right)^{\frac d2}e^{-\frac{|k|^2}{4z}-\frac{i(a\cdot k)}{2z}+\frac{|a|^2}{4z}}.
\end{eqnarray*}

\end{proof}

\bibliographystyle{siam}

\vspace{2cm}

\noindent \textcolor{gray}{$\bullet$} E. Faou --  INRIA Rennes, Univ Rennes \& Institut de Recherche Math\'ematiques de Rennes, CNRS UMR 6625 Rennes, Campus Beaulieu F-35042 Rennes Cedex, France.\\
{\it E-mail}: erwan.faou@inria.fr

\noindent \textcolor{gray}{$\bullet$} A. Mouzard --  ENS de Lyon, CNRS, Laboratoire de Physique, F-69342 Lyon, France.\\
{\it E-mail}: antoine.mouzard@math.cnrs.fr

\end{document}